\documentclass[11pt,dvipsnames]{amsart}

\usepackage{graphicx}
\usepackage{multirow}
\usepackage{amsmath,amssymb,amsfonts,amsthm}
\usepackage{mathrsfs}
\usepackage[title]{appendix}
\usepackage{xcolor}
\usepackage{textcomp}
\usepackage{manyfoot}
\usepackage{booktabs}
\usepackage{soul}

\usepackage{latexsym,  array, calrsfs,  hyperref, mathtools, multirow, latexsym, subfig, tikz-cd, verbatim, fge, xspace}
\usepackage{cleveref}
\usepackage{csquotes}
\usepackage{longtable}
\usepackage{mathpazo}
\usepackage{subfiles}

\usepackage{yfonts}
\usepackage[T1]{fontenc}

\usepackage{bm}

\usepackage[colorinlistoftodos, bordercolor=orange, backgroundcolor=orange!20, linecolor=orange, textsize=scriptsize]{todonotes}
\newtheorem{theorem}{Theorem}[section]
\newtheorem{lemma}[theorem]{Lemma}
\newtheorem{proposition}[theorem]{Proposition}
\newtheorem{corollary}[theorem]{Corollary}

\newtheorem{innercustomthm}{Theorem}
\newenvironment{customthm}[1]
  {\renewcommand\theinnercustomthm{#1}\innercustomthm}
  {\endinnercustomthm}

\newtheorem*{theorem*}{Theorem}

\theoremstyle{definition}

\newtheorem{example}[theorem]{Example}
\newtheorem{remark}[theorem]{Remark}
\newtheorem*{assumption*}{Assumption}

\newtheorem{definition}[theorem]{Definition}  

\AddToHook{env/example/begin}{\crefalias{theorem}{example}}

\crefname{example}{example}{examples}
\Crefname{example}{Example}{Examples}

\numberwithin{equation}{section}
\numberwithin{table}{section}
\numberwithin{figure}{section}

\newcommand{\C}{\mathbb{C}}

\newcommand{\N}{\mathbb{N}}

\newcommand{\Q}{\mathbb{Q}}

\newcommand{\Z}{\mathbb{Z}}

\DeclareMathAlphabet{\pazocal}{OMS}{zplm}{m}{n}

\newcommand{\pB}{\pazocal{B}}
\newcommand{\pC}{\pazocal{C}}

\newcommand{\pE}{\pazocal{E}}
\newcommand{\pF}{\pazocal{F}}
\newcommand{\pG}{\pazocal{G}}
\newcommand{\pH}{\pazocal{H}}
\newcommand{\pI}{\pazocal{I}}

\newcommand{\pS}{\pazocal{S}}
\newcommand{\pZ}{\pazocal{Z}}

\newcommand{\pNB}{\pazocal{NB}}
\newcommand{\pV}{\pazocal{V}}

\newcommand{\fa}{\mathbf a}
\newcommand{\fb}{\mathbf b}
\newcommand{\fc}{\mathbf c}
\newcommand{\fd}{\mathbf d}

\newcommand{\fk}{\mathbf k}

\newcommand{\fs}{\mathbf s}
\newcommand{\ft}{\mathbf t}

\newcommand{\fv}{\mathbf v}
\newcommand{\fw}{\mathbf w}
\newcommand{\fx}{\mathbf x}
\newcommand{\fy}{\mathbf y}
\newcommand{\fz}{\mathbf z}

\newcommand{\fA}{\mathfrak{A}}
\newcommand{\fB}{\mathfrak{B}}

\newcommand{\fS}{\mathfrak{S}}

\newcommand{\Set}{\mathsf{Set}}
\newcommand{\MatIso}{\mathsf{Mat}}

\newcommand{\ps}[2]{(#1,#2)}

\newcommand{\Gbip}{G_{\text{bip}}}
\newcommand{\Gline}{G_{\text{line}}}
\newcommand{\Grel}{G}
\newcommand{\pred}[4]{V(#3, #4 \, | \, #1, #2)}

\newcommand{\ssets}[2]{#1(#2)}
\newcommand{\pssets}[2]{\ssets{#1_{\bullet}}{#2}}

\newcommand{\sSets}[1]{\ssets{\pS}{#1}}
\newcommand{\psSets}[1]{\pssets{\pS}{#1}}

\newcommand{\Aut}{\operatorname{Aut}}
\newcommand{\NB}{\pazocal{N}\!\pazocal{B}}
\newcommand{\cl}{\operatorname{cl}}

\newcommand{\rk}{\operatorname{rk}}
\newcommand{\rel}{\operatorname{rel}}

\newcommand{\suchthat}{\, : \,}

\newcommand{\Sym}{\mathrm{Perm}}

\newcommand{\QMat}[3]{\mathfrak{G}_{\bullet}(#1,#2,#3)}
\newcommand{\GS}[3]{\mathfrak{H}(#1,#2,#3)}
\newcommand{\QAlg}[2]{\mathfrak{G}_{\bullet}(#1,#2)}
\newcommand{\QAut}[2]{\mathfrak{Aut}_{#1}(#2)}
\newcommand{\QAutp}[2]{\mathfrak{Aut}_{\bullet}(#1,#2)}

\renewcommand{\setminus}{\mathbin{\fgebackslash}}

\DeclareMathOperator{\com}{com}

\makeatletter
\DeclareRobustCommand{\cev}[1]{%
  \mathpalette\do@cev{#1}%
}
\newcommand{\do@cev}[2]{%
  \fix@cev{#1}{+}%
  \reflectbox{$\m@th#1\vec{\reflectbox{$\fix@cev{#1}{-}\m@th#1#2\fix@cev{#1}{+}$}}$}%
  \fix@cev{#1}{-}%
}
\newcommand{\fix@cev}[2]{%
  \ifx#1\displaystyle
    \mkern#23mu
  \else
    \ifx#1\textstyle
      \mkern#23mu
    \else
      \ifx#1\scriptstyle
        \mkern#22mu
      \else
        \mkern#22mu
      \fi
    \fi
  \fi
}
\makeatother

\usepackage{adjustbox}

\newcommand{\cZ}{\mathcal{Z}}
\newcommand{\cH}{\mathcal{H}}

\usepackage[margin=1.2in]{geometry}

\begin{document} 

\title{Matroid isomorphism games}

\author[D. Corey]{Daniel Corey}
\author[S. Schmidt]{Simon Schmidt}
\author[M. Wack]{Marcel Wack}

\address{Daniel Corey, Embry--Riddle Aeronautical University, Department of Mathematics. Daytona Beach, Florida.}
\email{\href{mailto:dcorey2814@gmail.com}{daniel.corey@unlv.edu}}

\address{Simon Schmidt, Ruhr University Bochum, Faculty for Computer Science, Universit\"ats-Strasse 150, 44801 Bochum}
\email{\href{mailto:s.schmidt@rub.de}{s.schmidt@rub.de}}

\address{Marcel Wack, Technische Universit\"at Berlin, Institut f\"ur Mathematik, Sekr. MA 6-2, Strasse des 17 Juni 136, 10623 Berlin}
\email{\href{mailto:wack@math.tu-berlin.de}{wack@math.tu-berlin.de}}

\date{\today}

\begin{abstract}
    We define and study a collection of matroid isomorphism games corresponding to various axiomatic characterizations of matroids. These are nonlocal games played between two cooperative players. Each game is played on two matroids, and the matroids are isomorphic if and only if the game has a perfect classical winning strategy. 
    We define notions of quantum isomorphism in terms of perfect quantum commuting strategies, and we find a pair of nonisomorphic matroids that are quantum isomorphic. We also give a purely algebraic characterization of quantum isomorphic matroids. Finally, we use this notion of quantum isomorphism to describe a new type of quantum automorphism group of a matroid and derive a sufficient condition for a matroid to have nonclassical quantum automorphism. 
\end{abstract}

\maketitle

\section{Introduction}
Matroids were defined by Whitney in the 1930s \cite{Whitney:1935} and they are combinatorial structures that record notions of dependence such as linear dependence in linear algebra and cycles in graphs. Since their founding, matroids have been discovered to underlie many important constructions in areas such as algebraic geometry, tropical geometry, and combinatorial optimization. In this paper, we study isomorphisms of matroids in the context of nonlocal games and quantum information theory. 

A well-known fact about matroids is that they admit numerous cryptomorphic definitions, and each of them may be used to define matroid isomorphisms. 
Every matroid $M$ starts with a finite ground set which we denote by $\ssets E M$. A common way to define a matroid is to specify a collection of subsets of $\ssets E M$ that satisfy certain axioms. Common axiomatic frameworks include independent sets $\ssets \pI M$, bases $\ssets \pB M$, circuits $\ssets \pC M$, flats $\ssets \pF M$ and hyperplanes $\ssets \pH M$. (We review all of these notions and other matroid terminology in \S\ref{sec:matroid-basics}.) 
Fix one of the axiomatic systems above and denote it by $\pS$. An isomorphism $\varphi \colon M\to N$ of matroids $M$ and $N$ consists of a bijective map $\varphi \colon \ssets E M \to \ssets E N$ such that $A\in \ssets \pS M$ if and only if $\varphi(A) \in \ssets \pS N$. A simple but important observation is that it does not matter which axiomatic framework $\pS$ is chosen in this definition, they all yield equivalent definitions of a matroid isomorphism. The choices for $\pS$ above are all instances of what we define in \S\ref{sec:matroid-iso-str} to be matroid isomorphism structures, and these include other attributes of matroids that are not as common, such as the nonbases $\NB(M)$ of a matroid. 

In quantum information theory, nonlocal games (also known as Bell inequalities in the physics literature) provide a mathematical way to study quantum entanglement. A two-party nonlocal game includes a referee and two players, Alice and Bob, who may agree upon a strategy before the game starts, but they are not allowed to communicate during the game. Alice and Bob can for example use a deterministic classical strategy, in which the outputs of the players are determined by the inputs of the referee. 
Our first discovery is that matroid isomorphism admits an interpretation in terms of perfect winning strategies of a nonlocal game. In fact, we define a collection of games, one for each matroid isomorphism structure. 
We describe a simplified version of our games here, the general construction is provided in \S\ref{sec:non-local-games}. 

Let $M$ and $N$ be two matroids, and fix an isomorphism structure $\pS$.  Denote by $\psSets M$ the set of pairs $\fa = (S_{\fa}, p_{\fa})$ where $S_{\fa} \in \ssets \pS M$ and $p_{\fa} \in S_{\fa}$. The matroid isomorphism game involves an exchange of pairs of elements of $\psSets M$. Scoring in this game requires a relation between elements of $\psSets M$.  We define $\rel_{\ssets \pS M} \colon \pssets \pS M \times \pssets \pS M \to \{0,1,2,3\}$ by
 \begin{equation*}
    \rel_{\ssets \pS M}(\fa , \fb) = 
    \begin{cases}
      0 & \text{if } S_{\fa} = S_{\fb} \text{ and } p_{\fa} = p_{\fb}, \\
      1 & \text{if } S_{\fa} \neq S_{\fb} \text{ and } p_{\fa} = p_{\fb} \\
      2 & \text{if } S_{\fa} = S_{\fb} \text{ and } p_{\fa} \neq p_{\fb}, \\
      3 & \text{if } S_{\fa} \neq S_{\fb} \text{ and } p_{\fa} \neq p_{\fb} ,
    \end{cases}
  \end{equation*}
and we define $\rel_{\pS(N)}$ in a similar way. 
The referee sends $\fa$ and $\fb$  in $\psSets M$ to Alice and Bob, respectively. Alice and Bob each send elements of $\psSets N$ back to the referee, say Alice returns $\fx$ and Bob returns $\fy$. They win if 
\begin{equation*}
    \rel_{\ssets \pS M}(\fa , \fb)  = \rel_{\ssets \pS N}(\fx, \fy).
\end{equation*}
We note that there is a way to understand this game in terms of the colored graph isomorphism game as defined in \cite{robersonschmidt}, see Remark \ref{rmk:colored+graph+iso}.

In order for us to be able to use this game to determine whether or not $M$ and $N$ are isomorphic, the matroid isomorphism structure $\pS$ needs to \textit{cover} the matroids $M$ and $N$ in the sense that, e.g., for each $p\in \ssets E M$ there is a $A\in \pS(M)$ such that $p\in A$. For any fixed matroid isomorphism structure of interest, this technical condition is satisfied by the vast majority of matroids, see \Cref{prop:matroid+cover}. 

\begin{customthm}{A}[cf. \Cref{thm:iso-iff-classical}]\label{thm:A}
    Let $M$ and $N$ be matroids and let $\pS$ be a matroid isomorphism structure that covers $M$ and $N$. 
    Then $M$ and $N$ are isomorphic if and only if the matroid isomorphism game has a perfect classical winning strategy. 
\end{customthm}

With this characterization of matroid isomorphisms in hand, we are able to relax the notion of isomorphism by allowing Alice and Bob to use \textit{quantum} strategies in order to win the game.  Such strategies allow for Alice and Bob to perform (local) measurements on a shared entangled state in order to improve their chances of winning. There are several different notions of quantum strategies, we focus on quantum commuting strategies which we review in \S\ref{sec:non-local-games}, and we say that two matroids are $\pS$--quantum isomorphic if the matroid isomorphism game (played using the isomorphism structure $\pS$) has a perfect quantum commuting strategy. In contrast to ordinary matroid isomorphisms, we get a different notion of quantum isomorphism depending on the matroid isomorphism structure used in the game. We find that there are matroids that are quantum isomorphic, but not isomorphic.

\begin{customthm}{B}[cf. \Cref{thm:noniso-quantum-ism}]\label{thm:B}
    There exist a pair of matroids that are quantum isomorphic but not isomorphic. 
\end{customthm}

These two matroids have rank $3$ and are defined on sets with 18 elements. They are derived from the Mermin--Peres magic square game \cite{Mermin}, which is a game that famously illustrates how quantum entanglement may produce perfect winning strategies when perfect classical strategies are not available. We review this game in \S\ref{sec:quantum-iso-matroids} in the broader context of nonlocal games on linear binary constraint systems as formulated by Cleve and Mittal in \cite{CleveMittal}.  
We describe a procedure for how to obtain quantum isomorphic matroids from perfect quantum strategies of some of these games. 

Starting in \S\ref{sec:algebra}, we describe purely algebraic characterizations of quantum isomorphisms of matroids. We define a $\ast$-algebra $\QMat M N \pS$ associated to matroids $M$, $N$ and matroid isomorphism structure $\pS$ that parameterizes the quantum isomorphisms between $M$ and $N$. 
More precisely, we prove in \Cref{thm:algebraic-quantum-iso} that the nonvanishing of such a $\QMat M N \pS$ ensures that the two matroids are quantum isomorphic. 
Using this algebraic description of quantum isomorphism, we show that some basic numerical data of a matroid must be preserved under quantum isomorphism. 

\begin{customthm}{C}[cf. \Cref{thm:gs+same+size}]\label{thm:C}
    Suppose $\pS$ is an isomorphism structure that covers the matroids $M$ and $N$.  If $M$ and $N$ are $\pS$-quantum isomorphic, then their ground sets have the same size. 
\end{customthm}  

Now, consider the $\ast$-algebra $\QMat M N \pS$ in the case where $N$ is equal to $M$. The resulting algebra admits a natural coproduct, and in particular, we obtain a new quantum automorphism group $\QAutp{M}{\pS}$ of $M$. 
This is a quantization of the automorphism group $\Aut(M)$ of $M$, which becomes more apparent when we view this group as a subgroup of permutations on $\pS_{\bullet}(M)$ induced by automorphisms of $M$.
Similar to quantum automorphism groups of graphs and other discrete structures, see for example \cite{Foldedcube}, we prove that $\QAutp{M}{\pS}$ admits a \textit{disjoint automorphism criterion} for detecting quantum symmetries. This uses the description of the underlying $\ast$-algebra of $\QAutp{M}{\pS}$ in terms of the quantum automorphism group of the relation colored graph described in \Cref{def:relation+colored+graph}. 

Quantum automorphism groups of matroids were first defined and studied in \cite{CoreyJoswigSchanzWackWeber}. The authors document which matroids with small ground set have noncommutative quantum automorphism groups (as they define).  Using the disjoint automorphism criterion, we are able to show that the $\QAutp{M}{\pS}$ provide new quantizations of $\Aut(M)$. 

\begin{customthm}{D}\label{thm:D}
    There are matroids $M$ such that the underlying algebra of $\QAutp M \pS$ is noncommutative but its quantum automorphism groups in \cite{CoreyJoswigSchanzWackWeber} are commutative. 
\end{customthm}

\section{Matroids and isomorphism structures}
\label{sec:matroids}
In this section, we review the various characterizations of matroids and what it means for two matroids to be isomorphic.  
We refer the reader to \cite{Oxley} for a comprehensive account of matroid theory.  

\subsection{Matroid basics}
\label{sec:matroid-basics}
There are many equivalent ways of prescribing a matroid. Our starting point is the characterization in terms of bases. Let $r$ be a nonnegative integer. 
A \textit{matroid} is a pair $M = (\ssets E M, \ssets \pB M)$ consisting of a finite set $\ssets E M$  and a nonempty set $\ssets \pB M$ of $r$-element subsets of $\ssets E M$ that satisfies the \textit{basis exchange axiom}: for each pair of distinct elements $A$ and $B$ in $\ssets \pB M$ and $a\in A\setminus B$, there is a $b\in B\setminus A$ such that $(A\setminus \{a\}) \cup \{b\}$ is in $\ssets \pB M$.  The set $\ssets E M$ is called the \textit{ground set} of $M$. An element of $\ssets \pB M$ is called a \textit{basis} $M$.  The \textit{nonbases} of $M$, which we denote by $\ssets \NB M$, are the $r$-element subsets of $E(M)$ which are not bases.

Let $M$ be a matroid. We describe several other characterizations of matroids. 
\begin{itemize}
    \item A subset $A\subseteq \ssets E M$ is \textit{independent} if $A$ is contained in a basis of $M$. The set of independent sets of $M$ is denoted by $\pI(M)$. 
    
    \item  A \textit{circuit} is a minimal dependent set of $M$. The set of circuits of $M$ is denoted by $\ssets{\pC}{M}$. 
    
    \item The \textit{rank function} of $M$ is the function $\rk_{M} \colon 2^{\ssets{E}{M}} \to \N_{0}$ defined by
    \begin{equation*}
        \rk_{M}(A) = \max(|X| \suchthat X\in \pI(M) \text{ and } X\subseteq A).
    \end{equation*}
    The rank of $M$ is $\rk_{M}(\ssets{E}{M})$ and is denoted by $\rk(M)$. This is the size of each basis of $M$. 

    \item The \textit{closure function} of $M$ is the function $\cl_{M} \colon 2^{\ssets{E}{M}} \to 2^{\ssets{E}{M}}$ defined by
    \begin{equation*}
        \cl_{M}(A) = \{b\in \ssets{E}{M} \suchthat \rk_{M}(A\cup \{b\}) = \rk_{M}(A)\}.
    \end{equation*}

    \item A subset $A \subset \ssets{E}{M}$ is a \textit{flat} if $\cl_{M}(A) = A$. The set of flats of $M$ is denoted by $\pF(M)$.

    \item A \textit{hyperplane} of $M$ is a flat whose rank is $r-1$ where $r$ is the rank of $M$.  The set of hyperplanes is denoted by $\cH(M)$. 
\end{itemize}

\noindent Each attribute in the above list admits its own axiomatic characterization that one may use to define a matroid.

\begin{example}
Let $E$ be a finite set and let $r$ be an integer satisfying $0\leq r \leq |E|$.  The \textit{uniform matroid} of rank $r$ on $E$ is the matroid $U(r,E)$ with where every $r$-element subset of $E$ is a basis. If $E = \{1,\ldots,n\}$, then we simply write $U(r,n)$.
\end{example}

Matroids may be obtained from vector configurations in the following way. Let $K$ be a field and let $V$ be an $r$-dimensional $K$-vector space. Let $Q = \{v_1, \ldots, v_n\}$ be a finite set of \textit{labeled} vectors in $V$ that linearly span $V$. Vectors may coincide but they are distinguished by their labels. The matroid of $Q$ is denoted by $M(Q)$ and is defined in the following way. The ground set of $M(Q)$ is $\{1,\ldots,n\}$ (the set of labels), and a subset $A\subseteq \{1,\ldots, n\}$ is a basis for this matroid if $\{v_i \suchthat i\in A\}$ is a basis of $V$.  A matroid of this form is said to be $K$-\textit{realizable}. 

Matroids may also be obtained from graphs. Let $G$ be a graph with vertex set $V$ and edge set $E$. The \textit{cycle matroid} of $G$ is the matroid whose ground set is the set of edges $E$, and a subset $A\subset E$ is a basis of this matroid if $A$ is the set of edges in a spanning forest of $G$. 

\begin{example} 
  Let $E = \{ 1, \ldots, 9\}$ be points placed in a $3 \times 3$ grid (cf. \Cref{fig:magic-square-matroid}) and take as nonbases the ground set elements that are horizontally or vertically aligned in this grid. Explicitly, 
\begin{equation*}
   \NB(M) := \{ 123,456,789,147,258,369\},
 \end{equation*}
 where we abbreviate, e.g., $\{1, 2, 3\}$ by $123$. 
This matroid has rank $3$ and is realizable over $\Q$ since $M=M(Q)$ where $Q$ is a vector configuration consisting of the columns of the matrix
  \label{ex:magic-square-matroid}
  \begin{equation*}
    \begin{pmatrix}
      1 & 0 & 1 & 0 & 2 & 2 & 1 & 1 & 1 \\
      0 & 1 & 1 & 0 & 5 & 5 & 0 & 3 & 2 \\
      0 & 0 & 0 & 1 & 2 & 6 & 4 & 1 & 2
    \end{pmatrix}.
  \end{equation*}
  This matroid is featured prominently in  \S\ref{sec:quantum-iso-matroids}.
\end{example}

\begin{figure}
    \centering
    \includegraphics[width=0.2\linewidth]{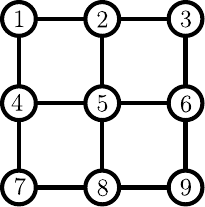}
    \caption{A rank 3 matroid}
    \label{fig:magic-square-matroid}
\end{figure}

\subsection{Matroid isomorphism structures}
\label{sec:matroid-iso-str}
Two matroids $M$ and $N$ are \textit{isomorphic} if there is a bijection $\varphi\colon \ssets{E}{M} \to E(N)$ such that $A\in \ssets{\pB}{M}$ if and only if $\varphi(A)\in \ssets{\pB}{N}$; such a $\varphi$ is an \textit{isomorphism} of $M$ and $N$. An isomorphism may be characterized in analogous ways for other axiomatic structures of matroids; we formalize this in the following way. 

Denote by $\MatIso$ the class of all matroids and $\Set$ the class of all sets. A \textit{matroid isomorphism structure} on $\MatIso$ is a function $\pS \colon \MatIso \to \Set$ satisfying
\begin{enumerate}
    \item for each matroid $M$ we have $\ssets \pS M \subseteq 2^{\ssets E M}$, and
    \item for each pair of matroids $M$ and $N$,  a function $\varphi \colon \ssets E M \to \ssets E N$ defines an isomorphism $M \to N$ if and only if:
    \begin{equation*}
        A\in \ssets \pS M \quad \text{if and only if} \quad \varphi(A) \in \ssets \pS N.
    \end{equation*}
\end{enumerate}
Examples of matroid isomorphism structures include bases $\ssets \pB M$, nonbases $\ssets \NB M$, independent sets $\ssets \pI M$, circuits $\ssets \pC M$, flats $\ssets \pF M$, and hyperplanes $\ssets \pH M$. 

An \textit{automorphism} of $M$ is an isomorphism of $M$ to itself and the set of all automorphisms of $M$ forms the \textit{automorphism group} of $\Aut(M)$. 
 
In the following section we define a matroid isomorphism game for each matroid isomorphism structure. In order for the existence of perfect deterministic winning strategies to characterize matroid isomorphism, we need the matroid isomorphism structure $\pS$ to satisfy the following property. 
We say that $\pS$ \textit{covers} the matroid $M$ if for each $p\in \ssets E M$ there is an element $A\in \pS(M)$ such that $p\in A$. As we see in \Cref{prop:matroid+cover}, our frequently used isomorphism structures cover most matroids of interest. 

In order to formulate this proposition, we must review some basic attributes of and operations on matroids. 
Let $M$ be a matroid. A \textit{loop} of  $M$ is an element of $\ssets{E}{M}$ not contained in any basis of $M$. A \textit{coloop} is an element of $\ssets{E}{M}$ contained in every basis of $M$, equivalently, an element not contained in any circuit of $M$. 

The dual of $M$ is the matroid $M^{\vee}$ that has the same ground set as $M$ and 
\begin{equation*}
    \pB(M^{\vee}) = \{A\subseteq \ssets E M \suchthat \ssets E M \setminus A \in \ssets \pB M\}.
\end{equation*}
Given $A\subset \ssets E M$, the \textit{restriction} $M|A$ is the matroid on $X$ whose independent sets are the independent sets of $M$ contained in $A$. The \textit{deletion} $M\setminus A$ is the restriction of $M$ to $\ssets E M \setminus A$. The contraction $M/A$ of $M$ by $A$ is the dual of  $(M^{\vee})|A$.

Given $e\notin \ssets{E}{M}$, the \textit{free extension} of $M$ by $e$ is the matroid $M + e$ whose ground set is $\ssets{E}{M + e} = \ssets{E}{M} \sqcup \{e\}$, and 
\begin{equation}
\label{eq:bases-free-ext}
    \ssets{\pB}{M+e} = \ssets{\pB}{M} \cup \left\{A\cup \{e\} \suchthat A\in \ssets{\pI}{M},\ |A| = \rk(M)-1 \right\}. 
\end{equation}

The matroid $M$ is \textit{paving} if each circuit of $M$ has size at least $\rk(M)$. 

Given two matroids $M$ and $N$, their \textit{direct sum} is the matroid $M\oplus N$ with
\begin{equation*}
    \ssets{E}{M \oplus N} = \ssets E M \sqcup \ssets E N, \quad \ssets{\pB}{M \oplus N} = \{A\cup B \suchthat A\in \ssets \pB M, B\in \ssets \pB N\}.
\end{equation*}
A matroid is \textit{connected} if it is not the direct sum of two matroids. 

\begin{proposition}
\label{prop:matroid+cover}
    Let $M$ be a matroid of rank $r$ whose ground set has $n$ elements. 
    \begin{enumerate}
        \item The set of bases $\ssets \pB M$ covers $M$ if and only if  $M$ has no loops. 
        \item The set of circuits $\pC(M)$ covers $M$ if and only if $M$ has no coloops. 
        \item The set of nonbases $\NB(M)$ covers $M$ if and only if $M$ is not a free extension of a paving matroid or $M$ does not split as a direct sum of the form $U(1,1) \oplus U(r-1,n-1)$.  
        \item The set of flats $\pF(M)$ and the set of hyperplanes $\pH(M)$ each cover the matroid $M$ unconditionally. 
    \end{enumerate}
\end{proposition}

\begin{proof}
    Statements (1), (2) and (4) are clear. Consider (3).  Suppose $M = N + e$ where $N$ is a paving matroid of rank $r$. This means that every $(r-1)$--element subset of $\ssets E N$ is an independent set of $N$. So every $r$-element subset of $\ssets E M$ containing $e$ is a basis by the characterization of free extension provided in \Cref{eq:bases-free-ext}.
    If $M \cong U(1,\{e\}) \oplus U(r-1, n-1)$, then every $r$-element subset containing $e$ is a basis of $M$. In either case, $\ssets \NB M$ does not cover $M$.
    
    Conversely, suppose $e\in \ssets E M$ is not contained in a nonbasis of $M$. In particular, every $(r-1)$--element subset of $\ssets E M \setminus \{e\}$ is independent. Let $N = M \setminus e$.
    If $\rk(N) = r-1$, then $e$ is a coloop of $M$ and every $(r-1)$--subset of $\ssets E N$ is a basis. Therefore $M \cong U(1,\{e\}) \oplus U(r-1,n-1)$. Now suppose that $\rk(N) = r$. Since every $(r-1)$--element subset of $\ssets E N$ is independent, we see that $N$ is a paving matroid. The identification of $M$ with $N + e$ may be seen using the basis characterization of free extension in  Equation \eqref{eq:bases-free-ext}.
\end{proof}

\section{Matroid isomorphism games and their strategies}
We begin this section with a brief review of nonlocal games and their properties. We also describe deterministic and quantum commuting strategies. Then we define the $(M,N,\pS)$--matroid isomorphism game for each pair of matroids $(M,N)$ and isomorphism structure $\pS$ that covers $M$ and $N$. We then show that $M$ and $N$ are isomorphic if and only if this game has a perfect deterministic strategy, and we give a criterion for perfect quantum commuting strategies to exist. The latter is useful in \S\ref{sec:algebra} where we describe a purely algebraic criterion for the existence of perfect quantum commuting strategies. 
\subsection{Reminders on non-local games}
\label{sec:non-local-games}
A standard reference for nonlocal games is \cite{CleveHoyerTonerWatrous}.
Our non-local games are played between two cooperating players Alice and Bob together with a referee. Alice and Bob each have their own input sets $X_{A}$ and $X_{B}$, respectively, and their own output sets $Y_{A}$ and $Y_{B}$, respectively. We assume that all of these sets are finite. Let $V \colon X_{A}\times X_{B} \times Y_{A} \times Y_B \to \{0,1\}$ be a predicate.
We write the value of $V$ on the tuple $(a,b,x,y)\in  X_{A}\times X_{B} \times Y_{A} \times Y_B$ by $\pred{a}{b}{x}{y}$.  
The rules of the non-local game $\pG = \pG(X_{A}, X_{B}, Y_{A}, Y_{B}, V)$ are as follows. 
The referee chooses at random (for us, the distribution will always be uniform) a pair of questions $(a,b)$ in $X_{A}\times X_{B}$. Alice responds with an answer $x \in Y_{A}$ and Bob responds with an answer $y \in Y_{B}$.  Alice and Bob win if $\pred{a}{b}{x}{y} = 1$; otherwise they lose. While Alice and Bob are not allowed to communicate during the game, they may strategize beforehand. 

The non-local game $\pG$ is \textit{synchronous} if $X_{A} = X_{B} = X$, $Y_{A} = Y_{B} = Y$, and $\pred{a}{a}{x}{y} = 0$ whenever $x\neq y$.  As defined in \cite{PaulsenRahaman}, $\pG$ is  \textit{bisynchronous} if it is synchronous and in addition $\pred{a}{b}{x}{x} = 0$ whenever $a\neq b$. 

A \textit{deterministic strategy} for $\pG$ consists of a pair of functions $f \colon X_{A} \to Y_{A}$ and $g \colon X_{B} \to Y_{B}$ such that $x = f(a)$ and $y = g(b)$. That is, Alice and Bob's responses are functions of their inputs. Such a strategy is \textit{perfect} if it guarantees victory, i.e., $\pred{a}{b}{f(a)}{g(b)} = 1$ for all $a\in X_{A}$ and $b \in X_{B}$.  If the game $\pG$ is synchronous, then $f = g$ holds in a perfect deterministic strategy. 

More generally, a classical strategy allows for Alice and Bob to exploit shared randomness in order to increase their chances of winning. When the input sets are finite (which is always the case in this paper), the existence of a perfect classical strategy ensures the existence of a perfect deterministic strategy. 

There are several different formulations of quantum strategies, and our focus is on the quantum commuting variation, see \cite[\S~2]{AtseriasEtAl} or \cite{PaulsenSeveriniStahlkeTodorovWinter}. 
We formulate this in the language of positive operator-valued measurements. 
Let $H$ be a complex Hilbert space. A \textit{positive operator-valued measurement}, POVM for short, consists of a \textit{finite} family $\pF = (F_{x})_{x\in X}$ of Hermitian positive semidefinite linear operators $F_x$ on $H$ such that 
\begin{equation*}
    \sum_{x\in X} F_x = I_{H}
\end{equation*}
where $I_{H}$ denotes the identity on $H$. They are called \textit{projective measurements}, or PVM's, if all operators $F_x$ in the POVM are projective, i.e., $F_x=F_x^*=F_x^2$.

A \textit{quantum commuting strategy} for the non-local game $\pG$ may consists of the following:
\begin{itemize}
    \item a complex Hilbert space $H$;
    \item a unit vector $\psi\in H$;
    \item families of POVMs $(\pE_{x})_{x\in Y_{A}}$ and $(\pF_{y})_{y\in Y_{B}}$  on $H$ for Alice and Bob, respectively;
    \item write $\pE_{x} = (E_{ax})_{a\in X_{A}}$ and $\pF_{y} = (F_{by})_{b\in X_{B}}$.  We require that each pair of operators $E_{ax}$ and $F_{by}$ commute.
\end{itemize}  
The probability of Alice and Bob answering with $x$ and $y$ on the input $a$ and $b$ is
\begin{equation*}
    p(x, y | a, b) = 
    \psi^{*}(E_{ax} F_{by}) \psi.
\end{equation*}

\subsection{The game}
Fix matroids $M$ and $N$ and an isomorphism structure $\pS$ that covers $M$ and $N$.
Given a matroid isomorphism structure $\pS$, define
\begin{equation*}
    \pS_{\bullet}(M) = \{(S, p) \suchthat S \in \ssets \pS M, \ p\in A \}.
\end{equation*}
Throughout, pointed sets will be denoted by lowercase bold letters. 
Given a pointed set $\fa = (S, p)$, we write $S_{\fa}$ for $S$ and $p_{\fa}$ for $p$.

To define an isomorphism game, we need a way to score the players' answers. To that end, define the map $\rel_{\pS(M)} : \pS_{\bullet}(M) \times \pS_{\bullet}(M)  \to \{0,1,2,3\}$ by
  \begin{equation*}
    \rel_{\ssets \pS M}(\fa , \fb) = \begin{cases}
      0 & \text{if } S_{\fa} = S_{\fb} \text{ and } p_{\fa} = p_{\fb} \\
      1 & \text{if } S_{\fa} \neq S_{\fb} \text{ and } p_{\fa} = p_{\fb} \\
      2 & \text{if } S_{\fa} = S_{\fb} \text{ and } p_{\fa} \neq p_{\fb} \\
      3 & \text{if }  S_{\fa} \neq S_{\fb} \text{ and } p_{\fa} \neq p_{\fb}
    \end{cases}
  \end{equation*}
To illustrate this relation, we associate to the pair $(M,\pS)$ an edge colored graph.
  
\begin{definition}
\label{def:relation+colored+graph}
The \textit{relation colored graph} of the pair $(M,\pS)$ is the graph $\Grel(M,\pS)$ whose vertex set is $\psSets M$, and edge set
\begin{equation*}
    E(\Grel(M,\pS)) = \left\{  \{\fa,\fb\} \suchthat \rel_{\pS(M)}(\fa,\fb) \in \{1, 2\} \right\}.
\end{equation*}
The color of the edge $\{\fa,\fb\}$ is $\rel_{\ssets \pS M}(\fa , \fb)$.
\end{definition}

\begin{remark}
There are other natural graphs associated to a matroid $M$ and matroid isomorphism structure $\pS$.  Consider the bipartite graph $\Gbip(M,\pS)$ with vertices $V = \ssets E M \cup \sSets M $ and edges 
    \[
      E = \{ \{a, A\} \suchthat A \in \sSets M, a \in A\}.
    \]
Then, the line graph $\Gline(M,\pS)$ of $\Gbip(M,\pS)$, i.e., the graph with vertices $E(\Gbip(M,\pS))$ and edges
    \[
      E = \{(\{A,a\},\{B,b\}) \suchthat A = B \text{ or } a=b,\ \text{ but not both}\}
    \]
is the uncolored version of the graph $\Grel(M,\pS)$ defined above.
\end{remark}

\begin{example}
  \label{ex:U23}
  Let $M = U(2,3)$ be the uniform matroid on 3 elements. We know from above that the set of bases $\ssets{\pB}{M} = \{ 12, 13, 23 \}$ (as usual, we abbreviate, e.g., $\{1,2\}$ by $12$) defines an isomorphism structure. The pointed bases are simply
  \begin{equation*}
  \pssets{\pB}{M} = \{ \ps{12}{1}, \ps{12}{2}, \ps{13}{1}, \ps{13}{3}, \ps{23}{2}, \ps{23}{3}\}.
  \end{equation*}
  The graphs $\Gbip(M,\pS)$, $\Gline(M,\pS)$ and the colored graph $\Grel(M,\pS)$ are illustrated in \Cref{fig:U23}. 
\end{example}
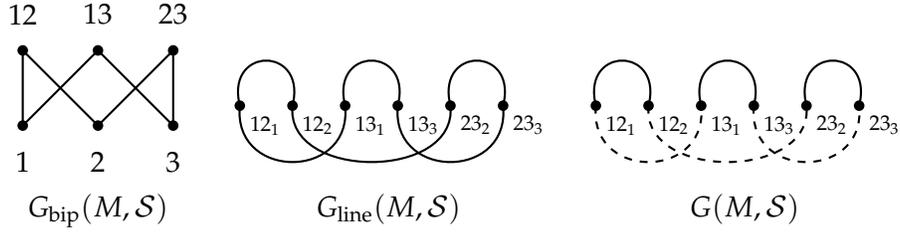
\begin{figure}[t]
\begin{tabular}{ccc}
\begin{tikzpicture}[scale=1]
  \fill (0,0) coordinate (12) circle (2pt) node[above=6pt] {12};
  \fill (1,0) coordinate (13) circle (2pt) node[above=6pt] {13};
  \fill (2,0) coordinate (23) circle (2pt) node[above=6pt] {23};
  \fill (0,0) ++(0,-1) coordinate (1) circle (2pt) node[below=6pt] {1};
  \fill (1,0) ++(0,-1) coordinate (2) circle (2pt) node[below=6pt] {2};
  \fill (2,0) ++(0,-1) coordinate (3) circle (2pt) node[below=6pt] {3};
  \draw[thick,solid] (12) -- (1);
  \draw[thick,solid] (12) -- (2);
  \draw[thick,solid] (13) -- (1);
  \draw[thick,solid] (13) -- (3);
  \draw[thick,solid] (23) -- (2);
  \draw[thick,solid] (23) -- (3);
\end{tikzpicture} &
\begin{tikzpicture}[scale=1]
  \fill (0,0) coordinate (a) circle (2pt) node[below right] {\scriptsize $12_1$};
  \fill (0.7,0) coordinate (b) circle (2pt) node[below right] {\scriptsize $12_2$};
  \fill (1.4,0) coordinate (c) circle (2pt) node[below right] {\scriptsize $13_1$};
  \fill (2.1,0) coordinate (d) circle (2pt) node[below right] {\scriptsize $13_3$};
  \fill (2.8,0) coordinate (e) circle (2pt) node[below right] {\scriptsize $23_2$};
  \fill (3.5,0) coordinate (f) circle (2pt) node[below right] {\scriptsize $23_3$};
  \draw[thick] (a) .. controls +(-0.2,0.8) and +(0.2,0.8) .. (b);
  \draw[thick] (c) .. controls +(-0.2,0.8) and +(0.2,0.8) .. (d);
  \draw[thick] (e) .. controls +(-0.2,0.8) and +(0.2,0.8) .. (f);
  \draw[thick] (a) .. controls +(-0.1,-1) and +(0.1,-1) .. (c);
  \draw[thick] (b) .. controls +(-0.1,-1) and +(0.1,-1) .. (e);
  \draw[thick] (d) .. controls +(-0.1,-1) and +(0.1,-1) .. (f);
\end{tikzpicture} &
\begin{tikzpicture}[scale=1]
  \fill (0,0) coordinate (a) circle (2pt) node[below right] {\scriptsize $12_1$};
  \fill (0.7,0) coordinate (b) circle (2pt) node[below right] {\scriptsize $12_2$};
  \fill (1.4,0) coordinate (c) circle (2pt) node[below right] {\scriptsize $13_1$};
  \fill (2.1,0) coordinate (d) circle (2pt) node[below right] {\scriptsize $13_3$};
  \fill (2.8,0) coordinate (e) circle (2pt) node[below right] {\scriptsize $23_2$};
  \fill (3.5,0) coordinate (f) circle (2pt) node[below right] {\scriptsize $23_3$};
  \draw[thick] (a) .. controls +(-0.2,0.8) and +(0.2,0.8) .. (b);
  \draw[thick] (c) .. controls +(-0.2,0.8) and +(0.2,0.8) .. (d);
  \draw[thick] (e) .. controls +(-0.2,0.8) and +(0.2,0.8) .. (f);
  \draw[thick, dashed] (a) .. controls +(-0.1,-1) and +(0.1,-1) .. (c);
  \draw[thick, dashed] (b) .. controls +(-0.1,-1) and +(0.1,-1) .. (e);
  \draw[thick, dashed] (d) .. controls +(-0.1,-1) and +(0.1,-1) .. (f);
\end{tikzpicture} \\%
$\Gbip(M,\pS)$&$\Gline(M,\pS)$&$\Grel(M,\pS)$
\end{tabular}
\caption{Constructing the relation colored graph of $ M =U(2,3)$ and $\pB (M)$. For brevity, we write, e.g., $ab_a$ for the pointed set $(\{a,b\},a)$. }

\label{fig:U23}
\end{figure}

The $(M,N,\pS)$--\textit{matroid isomorphism game} is a two-player cooperative nonlocal game defined in the following way. 
The input and output set for both Alice and Bob is $\pS_{\bullet}(M) \sqcup \pS_{\bullet}(N)$. The referee chooses uniformly at random elements $\fa$ and $\fb$ of $\pS_{\bullet}(M) \sqcup \pS_{\bullet}(N)$ to give to Alice and Bob, respectively. Alice returns $\fx$ and Bob returns $\fy$, both in $\pS_{\bullet}(M) \sqcup \pS_{\bullet}(N)$. Alice and Bob win, i.e., $V(\fx, \fy | \fa, \fb) = 1$, if the following two conditions are met. 
The first condition is that $\fa \in \pS_{\bullet}(M)$ if and only if $\fx \in \pS_{\bullet}(N)$ and $\fb \in \pS_{\bullet}(M)$ if and only if $\fy \in \pS_{\bullet}(N)$, i.e., the players have to answer with a pointed set of the matroid they did not receive an input from. Provided this condition  is met,  Suppose $\{\fc, \fv \} =  \{\fa, \fx\}$ where  $\fc \in \pS_{\bullet}(M)$ and $\fv \in \pS_{\bullet}(N)$ and let $\{ \fd, \fw \} = \{\fb, \fy\}$ be defined in a similar way. The second condition is 
\begin{equation*}
    \rel_{\ssets \pS M}(\fc, \fd) = \rel_{\ssets \pS N}(\fv , \fw).
\end{equation*}

\begin{remark}
\label{rmk:colored+graph+iso}
Note that the $(M,N,\pS)$-matroid isomorphism game is the same as the colored graph isomorphism game, as defined in \cite{robersonschmidt}, of the relation colored graphs $\Grel(M,\pS)$ and $\Grel(N,\pS)$. Therefore, one could use this connection to obtain our next results from known facts about graph isomorphism games \textit{and} how the colored graph isomorphisms are related to matroid isomorphisms. In this section, our approach is to directly work with matroid isomorphisms and their interplay with strategies for the matroid isomorphism game. We make use of the relation colored graphs in \Cref{sec:disj-aut}.   
\end{remark}

\begin{proposition}\label{prop:matroid+bisynchronous}
    The $(M, N, \pS)$--matroid isomorphism game is bisynchronous.
\end{proposition}
\begin{proof}
Suppose both Alice and Bob receive $\fa$, and assume  without loss of generality that $\fa\in \pssets \pS M$. If Alice and Bob return $\fx, \fy \in \pssets \pS N$ (respectively) with $\fx \neq \fy$, then $\rel_{\pS(M)}(\fa, \fa) = 0$ and $\rel_{\pS(N)}(\fx,\fy) \neq 0$. This shows that the game is synchronous and a similar argument shows that this game is bisynchronous. 
\end{proof}

Because this game is synchronous, a perfect classical deterministic strategy is defined by a single function used by both Alice and Bob. 
Suppose $\varphi\colon M \to N$ is an isomorphism with inverse $\psi\colon N \to M$. One particular perfect deterministic strategy is  given by the function 
\begin{equation*}
    \Phi\colon \pssets \pS M  \sqcup \pssets \pS N \to \pssets \pS M \sqcup \pssets \pS N
\end{equation*}
where $\Phi(\fa) = (\varphi(S_{\fa}), \varphi(p_{\fa}))$ if $\fa \in \psSets M$ or $\Phi(\fa) = (\psi(S_{\fa}), \psi(p_{\fa}))$ if $\fa \in \psSets N$. Such a perfect deterministic strategy is said to be \textit{induced} by a matroid isomorphism. 

\begin{theorem}
\label{thm:iso-iff-classical}
	Suppose $M$ and $N$ are matroids and $\pS$ is a matroid isomorphism structure that covers $M$ and $N$. Every perfect deterministic strategy of the $(M, N, \pS)$--matroid isomorphism game is induced by a matroid isomorphism. 
    In particular, $M$ and $N$ are isomorphic if and only if the $(M, N, \pS)$--matroid isomorphism game has a perfect deterministic strategy.
\end{theorem}

\noindent First, we collect some properties of a perfect deterministic strategy in the following lemma. 

\begin{lemma}
\label{lem:perfect-classical}
Suppose $\Phi\colon \pssets{\pS}{M} \sqcup \pssets{\pS}{N} \to  \pssets{\pS}{M} \sqcup \pssets{\pS}{N}$ defines a perfect deterministic strategy. We have the following. 
\begin{enumerate}
\item The restriction of $\Phi$ to $\pssets \pS M$ is a function $\Theta \colon \pssets{\pS}{M} \to \pssets{\pS}{N}$ and the restriction of $\Phi$ to $\pssets \pS N$ is a function $\Xi \colon \pssets{\pS}{N} \to \pssets{\pS}{M}$. Furthermore, $\Theta$ and $\Xi$ are inverses of each other. 
\item There are functions $\vartheta \colon \ssets E M \to E(N)$ and $\xi \colon E(N) \to \ssets E M$ that make the following diagrams commute:
\begin{equation*}
\begin{tikzcd}
 \psSets M \arrow{r}{\Theta} \arrow{d}{} & \psSets N \arrow{d}{} \\%
\ssets{E}{M} \arrow{r}{\vartheta}& \ssets{E}{N}
\end{tikzcd} 
\quad \text{and} \quad
\begin{tikzcd}
 \psSets N \arrow{r}{\Xi} \arrow{d}{} & \psSets M \arrow{d}{} \\%
\ssets{E}{N} \arrow{r}{\xi}& \ssets{E}{M}
\end{tikzcd} 
\end{equation*}
The vertical arrows in the above diagram are defined by sending $\fa$ to $p_{\fa}$. Furthermore, $\vartheta$ and $\xi$ are inverses of each other. 
\end{enumerate}
\end{lemma}

\begin{proof}
    Consider (1). In a perfect deterministic strategy, Alice must return an element of $\pssets \pS N$ if and only if she receives an element of $\pssets \pS M$, and similarly for Bob. This proves that the restricted functions $\Theta$ and $\Xi$ are well defined. Next, we have
    \begin{equation*}
        \rel_{\pS(M)}(\Xi(\Theta(\fa)), \fa) = \rel_{\pS(N)}(\Theta(\fa), \Theta(\fa)) = 0
    \end{equation*}
    and so $\Xi \circ \Theta$ is the identity. Similarly, we have that $\Theta \circ \Xi$ is also the identity, and hence $\Theta$ and $\Xi$ are inverses of each other. 

    Now consider (2). Suppose $\fa$ and $\fb$ are in $\pssets \pS M$ such that $p_{\fa} = p_{\fb}$, $\Theta(\fa) = \fx$, and $\Theta(\fb) = \fy$. Since 
\begin{equation}
\label{eq:abxy}
     \rel_{\ssets \pS M}(\fa, \fb) = \rel_{\ssets \pS N}(\fx, \fy)
\end{equation}
and $\rel(\fa, \fb) \in \{0,1\}$,  we must also have $\rel(\fx, \fy) \in \{0,1\}$. Therefore $p_{\fx} = p_{\fy}$. Together with the fact that $\pS$ covers $M$, we have a well defined function $\vartheta$ that makes the above diagram commute. We have in a similar way the function $\xi$. 

Next we claim that $\vartheta$ and $\xi$ are inverses of each other. Let $p\in \ssets E M$ and $q=\vartheta(p)$.  Given $\fa \in \pssets{\pS}{M}$ and $\fy \in \pssets{\pS}{N}$ with $p_{\fa} = p$ and $p_{\fy} = q$, we have $\Theta(\fa) = \fx$ and $\Xi(\fy) = \fb$ with $p_{\fx} = \vartheta(p)$ and $p_{\fb} = \xi(q)$. We have the same equality as in Equation \eqref{eq:abxy}
and it's value is $0$ or $1$ since $\vartheta(p) = q$. So $p = (\xi\circ \vartheta)(q)$, i.e., $\xi\circ \vartheta$ is the identity. A similar argument shows that $\vartheta \circ \xi$ is also the identity, so $\vartheta$ and $\xi$ are inverses of each other. 
\end{proof}

\begin{proof}[Proof of \Cref{thm:iso-iff-classical}]
Suppose a perfect deterministic strategy exists.  
By  \Cref{lem:perfect-classical}(1), this defines a pair of functions $\Theta \colon \pssets \pS M \to \pssets \pS N$ and $\Xi \colon \pssets \pS N \to \pssets \pS M$ that are inverses to each other. Let  $\vartheta$ and $\xi$ be the functions defined in part (2) of that lemma. These functions satisfy $\Theta(\fa) = \fx$ and $\Xi(\fy) = \fb$ where $p_{\fx} = \vartheta(p_{\fa})$ and $p_{\fb} = \xi(p_{\fy})$.  We must show that  $\vartheta$ is an isomorphism, i.e., that $A \in \pS(M)$ if and only if $ \vartheta(A) \in \pS(N)$. Suppose $A = \{s_1, \ldots, s_k\} \in \ssets{\pS}{M}$, and denote by $\fa_i = (A, s_i)$.  Set
\begin{equation*}
    \fx_i = \Theta(\fa_i), \quad \text{for} \quad 1\leq i \leq k.
\end{equation*}  
For $j\geq 2$, we have that $\rel(\fa_1, \fa_j) = 2$, and hence $\rel(\fx_{1}, \fx_{j}) = 2$. So the $S_{\fx_j}$'s are all equal to each other; call this set $X$. Because $t_{i}=\vartheta(s_i)$, we see that $X$ contains $\vartheta(A)$. 

Suppose $X$ properly contains $\vartheta(A)$, say  $X=\{t_1, \ldots, t_k, t_{k+1}, \ldots, t_{\ell}\}$ where  $\ell > k$. Set $\fy_{i} = (X,t_i)$ and $\fb_i =  \Xi(\fx_i)$ for $k+1 < i \leq \ell$. Because $\xi$ is injective, we have that $\xi(t_i) \notin A$, and so $A\neq S_{y_{k+1}}$. But $\rel_{\ssets \pS N}(\fx_{k}, \fy_{k+1}) = 2$ implies that $\rel_{\ssets \pS M}(\fa_k, \fb_{k+1}) = 2$, so $A = S_{\fy_{k+1}}$, which is a contradiction. This proves that $X=\vartheta(A)$.  A similar argument applied to $\xi$ shows that if $\vartheta(A) \in \ssets \pS N$ then $A\in \ssets \pS M$. 
\end{proof}

\begin{corollary}
\label{cor:induced+iso}
    If $\Theta \colon \psSets M \to \psSets N$ is a bijection satisfying 
    \begin{equation*}
        \rel_{\pS(M)}(\fa, \fb) = \rel_{\pS(N)}(\Theta(\fa), \Theta(\fb))
    \end{equation*}
    for all $\fa,\fb\in \psSets M$, then $\Theta$ is induced by a matroid isomorphism $M \to N$.
\end{corollary}

\begin{proof}
        The function $\Theta \sqcup \Theta^{-1} \colon \psSets M \sqcup \psSets N \to \psSets M \sqcup \psSets N$ defines a perfect deterministic strategy for the $(M,N,\pS)$--matroid isomorphism game. By \Cref{thm:iso-iff-classical}, this strategy must be induced by a matroid isomorphism $\varphi \colon M \to N$. 
\end{proof}

\subsection{Perfect quantum strategies for the matroid isomorphism game} 
As before, let $M$ and $N$ be two matroids and let $\pS$ be an isomorphism structure that covers $M$ and $N$.  With  \Cref{thm:iso-iff-classical} in mind, we say that $M$ and $N$ are $\pS$-\textit{quantum isomorphic} if the $(M,N,\pS)$--isomorphism game has a perfect quantum commuting strategy. In this subsection, we give a characterization of $\pS$-quantum isomorphic matroids that we use in \S\ref{sec:algebra}.
First, we review the following general fact about synchronous nonlocal games, see \cite[Cor.~5.6]{PaulsenSeveriniStahlkeTodorovWinter}.

\begin{lemma}\label{lem:synchronous}
Let $\pG$ be a synchronous nonlocal game with both input sets equal to $X$ and both output sets equal to $Y$. There exists a perfect quantum commuting strategy if and only if there is a $C^*$-algebra $A$, PVM's $(F_{ax})_{x\in Y} \subseteq A$ for all $a\in  X$, and a faithful tracial state $\tau \colon A\to \C$ such that
\begin{align}\label{eq:synch}
    \tau(F_{ax} F_{by}) = 0 \quad \text{if} \quad \pred{a}{b}{x}{y} = 0.
\end{align}
\end{lemma}

We have the following characterization of when the $(M,N,\pS)$--isomorphism game has a perfect quantum strategy. The statement and its proof are similar to \cite[Thm 5.4/5.13]{AtseriasEtAl}.

\begin{theorem}\label{thm:perfectstrat}
The $(M,N,\pS)$--isomorphism game has a perfect quantum commuting strategy if and only if there is a $C^*$-algebra $A$, projections $F_{\fa\fx} \in A$ for $\fa\in \psSets M$ and $\fx \in \psSets N$, and a faithful tracial state $\tau \colon A\to \C$ such that 
\begin{enumerate}
        \item $\sum_{\ft \in \psSets N} F_{\fa \ft} = 1$ for each $\fa \in \psSets M$, 
        \item $\sum_{\fs \in \psSets M} F_{\fs \fx} = 1$ for each $\fx \in \psSets N$, 
        \item $F_{\fa \fx} F_{\fb \fy} = 0$ if $\rel_{\pS(M)}(\fa, \fb) \neq \rel_{\pS(N)}(\fx, \fy)$,
    \end{enumerate}
    and $p(\fx, \fy | \fa, \fb) = \tau(F_{\fa\fx} F_{\fb\fy})$.
\end{theorem}

\begin{proof}
Recall that for the $(M,N,\pS)$--isomorphism game we have $X=Y=\pS_{\bullet}(M) \sqcup \pS_{\bullet}(N)$ and consider the PVM's $(F_{\fa\fx})_{\fx\in Y}$ obtained from \Cref{lem:synchronous}. We get that (3) is immediate, since $\tau$ is faithful. For (1), consider the case $\fa, \fx\in \psSets M$ and recall that the players always lose if one of them answers with a ground set element of the same matroid the player got a question from. Therefore
\begin{align*}
    \tau(F_{\fa\fx})=\tau(F_{\fa\fx}\sum_\fy F_{\fb\fy})=\sum_\fy \tau(F_{\fa\fx}F_{\fb\fy})=0
\end{align*}
by \Cref{eq:synch}, which yields $F_{\fa\fx}=0$ since $\tau$ is faithful. This shows (1), because 
\begin{equation*}
\sum_{\ft\in \psSets N} F_{\fa \ft}=\sum_{\ft\in Y} F_{\fa \ft}= 1
\end{equation*}
for each $\fa \in \psSets M$.

It remains to show (2) for the first direction. Note that $F_{\fa\fx}F_{\fx\fa'}=0$ for $\fa\neq \fa'$ since $V(\fx,\fa,\fa',\fx)=0$ in the matroid isomorphism game. Therefore
\begin{align*}
    F_{\fa\fx}=F_{\fa\fx}\sum_{\fa'}F_{\fx\fa'}=F_{\fa\fx}F_{\fx\fa}=\sum_{\tilde \fa}F_{\tilde{\fa}\fx}F_{\fx\fa}=F_{\fx\fa}.
\end{align*}
We deduce $\sum_{\fs\in \psSets M} F_{\fs \fx}=\sum_{\fs\in \psSets M} F_{\fx\fs}=\sum_{\fs\in Y} F_{\fx\fs}= 1$ for $\fx \in \psSets N$, where we also use $F_{\fx\fs}=0$ for $\fs\in \psSets N$. This shows (2).

For the other direction, suppose we have projections and a tracial state as in the theorem statement. Define $F_{\fa\fx}=F_{\fx\fa}$ and let $F_{\fa\fx}=0$ for $\fa, \fx\in \psSets M$ and $\fa, \fx\in \psSets N$. It is easy to see that $(F_{\fa\fx})_{\fx\in Y}$ are PVM's and (3) yields \Cref{eq:synch}. This completes the proof. 
\end{proof}

\begin{remark}
In a $C^*$-algebra $A$, projections $(P_x)_x\subseteq A$ such that $\sum_xP_x=1$ always fulfill the relation $P_xP_{y}=0$ for all $x \neq y$. Note that this is in general not true in a $\ast$-algebra, which is why we have the additional relations $P_xP_{y}=0$ for $x\neq y$. 
\end{remark}

\section{Nonisomorphic matroids that are quantum isomorphic}
\label{sec:quantum-iso-matroids}
We produce a pair of matroids that are quantum isomorphic but not isomorphic. This pair of matroids are defined on a ground set with 18 elements and have rank 3. More generally, we describe how to obtain a matroid from a special type of binary constraint system. 

\subsection{Matroids and binary constraint systems}
\label{sec:matroids+BCS}
A \textit{linear binary constraint system (LBCS)}, as defined in \cite{CleveLiuSlofstra, CleveMittal}, consists of a pair $(\pV, \pE)$  where $\pV$ is a set of $m$ binary-valued variables and $\pE$ is a set of $n$ linear equations in these variables with binary coefficients. The most famous examples come from the Mermin--Peres magic square \cite{Mermin}, which consists of the $m=9$ variables $x_1,\ldots, x_9$ and the $n=6$ equations
\begin{align}\label{magicsquare}
\begin{array}{ccc}
     x_1 x_2 x_3 = s_1,  &
     x_4 x_5 x_6 = s_2,  &
     x_7 x_8 x_9 = s_3, \\
     x_1 x_4 x_7 = s_4,  &
     x_2 x_5 x_8 = s_5,  &
     x_7 x_6 x_9 = s_6
\end{array} 
\end{align}
where each $s_i$ is $1$ or $-1$. (We write the linear equations multiplicatively so that we may use $\{-1,1\}$--valued variables instead of $\{0,1\}$--valued variables.)

First, given a matroid $M$, let us describe how to obtain a LBCS. Our prototype for this procedure is given by the way one obtains the magic square game from the matroid in \Cref{ex:magic-square-matroid}. The variables of the magic square game correspond to the ground set elements of this matroid and the equations correspond to the nonbases of this matroid. Actually, we find it more convenient to view the nonbases of this matroid as special kinds of hyperplanes of a matroid, i.e., cyclic hyperplanes, which we now describe. 

Recall from \S \ref{sec:matroids} that a hyperplane of $M$ is a flat of rank $\rk(M) - 1$ and the set of all hyperplanes is denoted by $\pH(M)$. A \textit{cyclic flat} is a flat that is a union of circuits, and a \textit{cyclic hyperplane} is a hyperplane that is also a cyclic flat. Denote by $\pZ(M)$ the set of cyclic flats of $M$ and $\pZ^{k}(M)$ the set of cyclic flats that have rank $\rk(M) - k$. In particular, $\pZ^{1}(M)$ consists of the set of cyclic hyperplanes of $M$. 

Similar to the collection of flats, the collection of cyclic flats forms a lattice. Given two cyclic flats $X$ and $Y$, their \textit{join} is $X\vee Y = \cl_{M}(X\cup Y)$ and their \textit{meet}, denoted by $X\wedge Y$, the union of all circuits contained in $X\cap Y$.
The lattice of cyclic flats, together with the restriction of the rank function to $\pZ(M)$, completely determine $M$. This leads to the following axiomatic characterization of matroids as formulated in \cite[Thm 3.2]{BonindeMier}.

\begin{theorem}
\label{thm:cyclic-flat-axioms}
    Let $E$ be a finite set,  $\pZ\subset 2^{E}$ a lattice (under inclusion), and $r\colon \pZ \to \N_{0}$ a function. There is a matroid $M$ whose cyclic flats are $\pZ$ and whose rank function restricted to $\pZ$ is $r$ if and only if
    \begin{enumerate}
        \item $r(0_{\pZ}) = 0$,
        \item $0 < r(Y) - r(X) < |Y\setminus X|$ for each pair $X,Y \in \pZ(M)$ with $X\subsetneq Y$, and
        \item $r(X) + r(Y) \geq r(X \vee Y) + r(X \wedge Y) + |(X\cap Y) \setminus (X\wedge Y)|$ for all $X,Y\in \pZ$. 
    \end{enumerate}
\end{theorem}

Fix a collection $S = \left( s_{H} \suchthat H\in \cZ(M) \right)$ in $\{\pm 1\}^{\cZ^{1}(M)}$.  
We associate to $(M,S)$ a LBCS $(\pV, \pE)$ by setting  
\begin{equation*}
	\pV =  \{x_a \suchthat a\in \ssets E M\}, \quad \text{and} \quad \pE = \left\{ \prod_{a\in H} x_a = s_{H} \suchthat H\in \cZ^{1}(M) \right\}.
\end{equation*}

Next, we describe how to obtain a new matroid from a LBCS obtained from this matroid $M$. Our construction requires that $M$ is a rank $3$ sparse paving matroid. A matroid is \textit{sparse paving} if both the matroid and its dual are paving matroids (recall that a matroid is paving if the size of each circuit is at least the rank of the matroid). The following proposition gives a clear description of rank 3 sparse paving matroids. For this, we need the definition of a \textit{simple} matroid, which is a matroid $M$ that has no loops and such that every subset $A$ that has rank 1 is a singleton set. 

\begin{proposition}
\label{prop:sparse-paving}
	If $M$ is a rank $3$ matroid, then $M$ is sparse paving if and only if 
	  $M$ is simple and each cyclic hyperplane of $M$ has exactly $3$ elements. 
\end{proposition}

\begin{proof}
We use the following fact from \cite[Prop.~2.1.6]{Oxley}: a subset $H\subseteq \ssets E M$ is a hyperplane of $M$ if and only if $\ssets E M \setminus H$ is a circuit of $M^{\vee}$. 

Suppose $M$ is sparse paving. 
Because $M$ is paving, every subset of size at most $2$ is independent; in particular $M$ is a simple matroid. 
Let $H\in \cZ^{1}(M)$. Since $H$ contains a circuit, it must have at least $3$ elements.  
By the fact stated above, $C = \ssets E M \setminus H$ is a circuit of $M^{\vee}$. Furthermore, $C$ has $|\ssets E M|-|H| \leq |\ssets E M|-3$ elements. 
As $M^{\vee}$ is also paving and has rank $|\ssets E M|-3$, we have $|C| \geq |\ssets E M| - 3$. 
So $C = |\ssets E M| - 3$, and hence $|H| = 3$. 

Conversely, suppose that $M$ is simple and each cyclic hyperplane of $M$ has exactly 3 elements. Because $M$ is simple and its rank is $3$, it is paving. If $C$ is a circuit of $M^{\vee}$, then $H = \ssets E M \setminus C$ is a hyperplane of $M$. By hypothesis, $H$ has either $2$ or $3$ elements. This implies that $|C| \geq |\ssets E M| - 3$. Because the rank of $M^{\vee}$ is $|\ssets E M| - 3$, we see that $M^{\vee}$ is paving. 
\end{proof}
 
We define a new matroid $M_{S}$ from this data by describing its cyclic flats. The ground set of $M_{S}$ is 
\begin{equation*}
	E(M_{S}) = \ssets E M \times \{\pm 1\}. 
\end{equation*}

Let $\pi\colon E(M_{S}) \to \ssets E M$ be the projection to the variable $\pi(a,s) = a$. 
Set $\pZ(M_{S}) = \{ \emptyset, E(M_{S}) \} \cup \pZ^{1}(M_{S})$ where
\begin{equation*}
    \pZ^{1}(M_{S}) = \left\{ K \subset E(M_{S}) \suchthat |\pi(K)| = |K|,\, \pi(K) \in \pZ^{1}(M),\, \text{ and }  \prod_{(a,s)\in K} s = s_{\pi(K)} \right\}.
\end{equation*}
Define $r\colon \pZ(M_{S}) \to \N_{0}$ by 
\begin{equation*}
    r(X) = 
    \begin{cases}
    0 & \text{ if } X = \emptyset, \\
    2 & \text{ if } X\in \pZ^{1}(M_{S}),  \\
    3 & \text{ if } X = E(M_{S}).
    \end{cases}
\end{equation*}
We view $\pZ(M_{S})$ as a lattice whose join and meet operations are
    \begin{equation*}
        X \vee Y = 
        \begin{cases}
        X & \text{ if } Y\subset X, \\
        Y & \text{ if } X\subset Y, \\
        E(M_{S}) & \text{otherwise},
        \end{cases}
        \qquad \text{and} \qquad
        X\wedge Y = 
        \begin{cases}
        X & \text{ if } X\subset Y, \\
        Y & \text{ if } Y\subset X, \\
        \emptyset & \text{otherwise}.
        \end{cases}
    \end{equation*}
\begin{theorem}
\label{thm:matroid-from-BCS}
	The lattice $\pZ(M_{S})$ and the function $r \colon \pZ(M_{S}) \to \N_{0}$ define a matroid $M_{S}$ as in  \Cref{thm:cyclic-flat-axioms}. 
\end{theorem}

\begin{proof}
This is a straightforward verification using \Cref{thm:cyclic-flat-axioms}. The most interesting case is the verification of the inequality in (3) for $X, Y\in \pZ^{1}(M_{S})$ with $X\neq Y$. In this case, $r(X) = r(Y) = 2$,  $r(X \vee Y) = 3$ and $r(X \wedge Y)=0$. Furthermore, $|X\cap Y|\leq 1$ because if $X\cap Y$ have 2 elements in common, they must be equal (this uses in an essential way the fact that each cyclic hyperplane of $M$ has exactly 3 elements). These sets satisfy the inequality in  (3), as required. 
\end{proof}

\subsection{Linear binary constraint system games}
\label{sec:LBCS+game}
As defined in \cite{CleveMittal}, one can associate a nonlocal game to every binary constraint system $(\pV, \pE)$. The game is played as follows (note that our game is a synchronous version of the games in \cite{CleveMittal}). The referee sends Alice and Bob equations $H_A$ and $H_B$ from the set $\pE$, respectively. They answer with $\{-1,1\}$--assignments to all variables in $H_A$ and $H_B$. They win the game if
\begin{itemize}
    \item[(i)] Their assignments are solutions to their equations,
    \item[(ii)] Their assignments agree on all variables that appear in both equations.
\end{itemize}

\noindent It is well-known \cite{Mermin}, that the nonlocal game associated to the magic square has a perfect quantum strategy, but no perfect classical strategy if $s_i=-1$ for an odd number of $i$ in \Cref{magicsquare}.

\subsection{Examples of quantum isomorphic matroids}

Let $M$ be a rank 3 sparse-paving matroid, $S = \left( s_{H} \suchthat H\in \cZ(M) \right)$ in $\{\pm 1\}^{\cZ^{1}(M)}$, and let $M_{S}$ be the matroid defined in \S~\ref{sec:matroids+BCS}. As a consequence of \Cref{prop:sparse-paving}, we have $\cZ^{1}(M) = \NB(M)$. So the set of nonbases of $M_{S}$ is given by
\begin{align*}
    \NB(M_{S}) = \left\{ \{(x_a,t_a), (x_b, t_b), (x_c, t_c)\} \suchthat \{x_a,x_b,x_c\}\in \NB(M),\ t_i\in\{\pm1\} \text{ and } t_at_bt_c=s_H \right\}, 
\end{align*}
i.e., a nonbasis consist of an equation of the linear system with fulfilling assignments. Note that by choosing different values for $s_H\in \{\pm 1\}$, we obtain (possibly) different matroids from $M$. If $s_H=1$ for all $H$, the linear system is \emph{homogeneous} and we call the associated matroid $M_{hom}$.
In the following, we write $(H,t)$ for the nonbasis of $M_{S}$ corresponding to $H\in\cZ(M)$ and $t\in\{\pm 1\}^H$.

\begin{theorem}
\label{thm:qisomorphic-matroid}
Consider a rank $3$ sparse paving matroid $M$, a collection $S = \left( s_{H} \suchthat H\in \cZ(M) \right)$ in $\{\pm 1\}^{\cZ^{1}(M)}$ and the associated linear system $(\pV, \pE)$ (see previous subsections). Assume there exists a perfect quantum strategy for the linear system game associated to $(\pV, \pE)$. Then $M_{hom}$ and $M_{S}$ are quantum isomorphic, i.e., the ($M_{hom}, M_{S}$, $\NB$)-isomorphism game has a perfect quantum strategy. 
\end{theorem}

\begin{proof}
The players Alice and Bob use the following strategy in the nonbasis isomorphism game: Upon receiving pointed nonbases $((H_A, t_A), (x, a))$ and $((H_B, t_B), (y,b))$, respectively, they play the linear system game (see \S\ref{sec:LBCS+game}) with questions $H_A$ and $H_B$ using the perfect quantum strategy. Since it is a perfect quantum strategy, they receive fulfilling assignments $k_A$ and $k_B$ to $H_A$ and $H_B$  which agree on the common variables in $H_A$ and $H_B$. The players then answer $((H_A, t_A*k_A), (x, ak))$ and $((H_B, t_B*k_B),(y,bl))$ in the isomorphism game, where $*$ denotes the entry-wise product and $k$ and $l$ are the assignments to $x$ and $y$ in $k_A$ and $k_B$. Note that since $t_A$ and $t_B$ are fulfilling assignments in the homogeneous system, we get that $t_A*k_A$ and $t_B*k_B$ fulfilling assignments for the LBCS associated to $S$, i.e., $(H_A, t_A*k_A)$ and $(H_B, t_B*k_B)$ are nonbasis in $M_{S}$. 

It remains to check that the relations are preserved in this strategy:

\medskip 
	
\noindent \textbf{Case 1:} $\rel(((H_A, t_A), (x, a)),((H_B, t_B), (y,b)))=0$. It is clear that the same is true for their answers, since the assignments they get agree on the common variables, which are all variables since $H_A=H_B$.
    
\medskip 
	
\noindent \textbf{Case 2:} $\rel(((H_A, t_A), (x, a)),((H_B, t_B), (y,b)))=1$. We have $(x,a)=(y,b)$ and since the answers agree on common variables, it holds $k=l$ and therefore $(x,ka)=(y,lb)$. We furthermore have $(H_A, t_A*k_A)\neq (H_B, t_B*k_B)$: either it holds $H_A\neq H_B$ or if $H_A=H_B$, then we have $t_A \neq t_B$ by assumption, and since we have $k_A=k_B$ it holds $t_A*k_A\neq t_B*k_B$.
    
\medskip 
	
\noindent \textbf{Case 3:} $\rel(((H_A, t_A), (x, a)), ((H_B, t_B), (y,b)))=2$. Since common variables get multiplied by the same number, we get that $(x, ak)\in (H_B, t_B*k_B)$ and $(y,bl)\in(H_A, t_A)$.
    
\medskip 
	
\noindent \textbf{Case 4:} $\rel(((H_A, t_A), (x, a)),((H_B, t_B), (y,b)))=3$. It suffices to check that if $(x, a)\notin (H_B, t_B)$, then $(x, ka)\notin (H_B, t_B*k_B)$ (and similarly for $(y,b)$). This is clear if $x\notin H_B$. If $x\in H_B$, note that assignments of common variables of $H_A$ and $H_B$ get multiplied by the same number. So if the assignments were different before, they are different after multiplying both with $k$. 
\end{proof}

We apply this theorem to find two matroids that are not isomorphic but are $\NB_{\bullet}$--quantum isomorphic. For the rest of this section, let $M$ be the matroid from   \Cref{ex:magic-square-matroid}, recall 
  \begin{equation*}
  \NB(M) = \{123,147,258,369,456,789 \} = \cZ^{1}(M).
  \end{equation*}
(As in that example, we write, e.g., $123$ for $\{1,2,3\}$.)
Define $S = (s_{H} \suchthat H\in \cZ^{1}(M))$ by $s_{789}=-1$ and $s_{H}=1$ for $H\in \cZ^{1}(M) \setminus \{789\}$. 
Let $P = M_{hom}$ and $Q = M_{S}$. 
These matroids have the same ground set $E$ where
\begin{equation*}
    E = \{(a,+1) \suchthat a\in [9] \} \cup \{(a,-1) \suchthat a\in [9] \},
\end{equation*}
and they each have exactly $24$ nonbases. 
Explicitly, the nonbases of $P$ are of the form 
\begin{align*}
    &\{ (a,+1), (b,+1), (c,+1) \}, \
    \{ (a,+1), (b,-1), (c,-1) \}, \\ 
    &\{ (a,-1), (b,+1), (c,-1) \}, \ 
    \{ (a,-1), (b,-1), (c,+1) \},
\end{align*}
where $\{a,b,c\} \in \NB(M)$. The nonbases of $Q$ consist of those listed above for $\{a,b,c\} \in \NB(M)$ except for $\{7,8,9\}$. Instead, $Q$ has as nonbases
\begin{align*}
    \{ (7,+1), (8,+1), (9,-1) \}, \
    &\{ (7,+1), (8,-1), (9,+1) \}, \\ 
    \{ (7,-1), (8,+1), (9,+1) \}, \
    &\{ (7,-1), (8,-1), (9,-1) \}.
\end{align*}
The relation colored graphs of these matroids are illustrated in \Cref{fig:matroids-PQ}.

\begin{figure}
  \centering
  \begin{minipage}{0.45\textwidth}
    \centering

\begin{tikzpicture}[x  = {(1cm,0cm)},
                    y  = {(0cm,1cm)},
                    z  = {(0cm,0cm)},
                    scale = 0.07,
                    color = {lightgray}]

  \coordinate (v0_unnamed__1) at (31.3775, 0);
  \coordinate (v1_unnamed__1) at (35.8775, -2.59808);
  \coordinate (v2_unnamed__1) at (35.8775, 2.59808);
  \coordinate (v3_unnamed__1) at (30.3083, 8.12109);
  \coordinate (v4_unnamed__1) at (35.3274, 6.77622);
  \coordinate (v5_unnamed__1) at (33.9825, 11.7953);
  \coordinate (v6_unnamed__1) at (27.1737, 15.6887);
  \coordinate (v7_unnamed__1) at (32.3698, 15.6887);
  \coordinate (v8_unnamed__1) at (29.7718, 20.1887);
  \coordinate (v9_unnamed__1) at (22.1872, 22.1872);
  \coordinate (v10_unnamed__1) at (27.2063, 23.5321);
  \coordinate (v11_unnamed__1) at (23.5321, 27.2063);
  \coordinate (v12_unnamed__1) at (15.6887, 27.1737);
  \coordinate (v13_unnamed__1) at (20.1887, 29.7718);
  \coordinate (v14_unnamed__1) at (15.6887, 32.3698);
  \coordinate (v15_unnamed__1) at (8.12109, 30.3083);
  \coordinate (v16_unnamed__1) at (11.7953, 33.9825);
  \coordinate (v17_unnamed__1) at (6.77622, 35.3274);
  \coordinate (v18_unnamed__1) at (0, 31.3775);
  \coordinate (v19_unnamed__1) at (2.59808, 35.8775);
  \coordinate (v20_unnamed__1) at (-2.59808, 35.8775);
  \coordinate (v21_unnamed__1) at (-8.12109, 30.3083);
  \coordinate (v22_unnamed__1) at (-6.77622, 35.3274);
  \coordinate (v23_unnamed__1) at (-11.7953, 33.9825);
  \coordinate (v24_unnamed__1) at (-15.6887, 27.1737);
  \coordinate (v25_unnamed__1) at (-15.6887, 32.3698);
  \coordinate (v26_unnamed__1) at (-20.1887, 29.7718);
  \coordinate (v27_unnamed__1) at (-22.1872, 22.1872);
  \coordinate (v28_unnamed__1) at (-23.5321, 27.2063);
  \coordinate (v29_unnamed__1) at (-27.2063, 23.5321);
  \coordinate (v30_unnamed__1) at (-27.1737, 15.6887);
  \coordinate (v31_unnamed__1) at (-29.7718, 20.1887);
  \coordinate (v32_unnamed__1) at (-32.3698, 15.6887);
  \coordinate (v33_unnamed__1) at (-30.3083, 8.12109);
  \coordinate (v34_unnamed__1) at (-33.9825, 11.7953);
  \coordinate (v35_unnamed__1) at (-35.3274, 6.77622);
  \coordinate (v36_unnamed__1) at (-31.3775, 0);
  \coordinate (v37_unnamed__1) at (-35.8775, 2.59808);
  \coordinate (v38_unnamed__1) at (-35.8775, -2.59808);
  \coordinate (v39_unnamed__1) at (-30.3083, -8.12109);
  \coordinate (v40_unnamed__1) at (-35.3274, -6.77622);
  \coordinate (v41_unnamed__1) at (-33.9825, -11.7953);
  \coordinate (v42_unnamed__1) at (-27.1737, -15.6887);
  \coordinate (v43_unnamed__1) at (-32.3698, -15.6887);
  \coordinate (v44_unnamed__1) at (-29.7718, -20.1887);
  \coordinate (v45_unnamed__1) at (-22.1872, -22.1872);
  \coordinate (v46_unnamed__1) at (-27.2063, -23.5321);
  \coordinate (v47_unnamed__1) at (-23.5321, -27.2063);
  \coordinate (v48_unnamed__1) at (-15.6887, -27.1737);
  \coordinate (v49_unnamed__1) at (-20.1887, -29.7718);
  \coordinate (v50_unnamed__1) at (-15.6887, -32.3698);
  \coordinate (v51_unnamed__1) at (-8.12109, -30.3083);
  \coordinate (v52_unnamed__1) at (-11.7953, -33.9825);
  \coordinate (v53_unnamed__1) at (-6.77622, -35.3274);
  \coordinate (v54_unnamed__1) at (0, -31.3775);
  \coordinate (v55_unnamed__1) at (-2.59808, -35.8775);
  \coordinate (v56_unnamed__1) at (2.59808, -35.8775);
  \coordinate (v57_unnamed__1) at (8.12109, -30.3083);
  \coordinate (v58_unnamed__1) at (6.77622, -35.3274);
  \coordinate (v59_unnamed__1) at (11.7953, -33.9825);
  \coordinate (v60_unnamed__1) at (15.6887, -27.1737);
  \coordinate (v61_unnamed__1) at (15.6887, -32.3698);
  \coordinate (v62_unnamed__1) at (20.1887, -29.7718);
  \coordinate (v63_unnamed__1) at (22.1872, -22.1872);
  \coordinate (v64_unnamed__1) at (23.5321, -27.2063);
  \coordinate (v65_unnamed__1) at (27.2063, -23.5321);
  \coordinate (v66_unnamed__1) at (27.1737, -15.6887);
  \coordinate (v67_unnamed__1) at (29.7718, -20.1887);
  \coordinate (v68_unnamed__1) at (32.3698, -15.6887);
  \coordinate (v69_unnamed__1) at (30.3083, -8.12109);
  \coordinate (v70_unnamed__1) at (33.9825, -11.7953);
  \coordinate (v71_unnamed__1) at (35.3274, -6.77622);

  \definecolor{blackish}{HTML}{001514}

  \tikzstyle{vertexstyle_unnamed__1_0} = [circle, scale=0.25, fill=blackish,label={[text=black, above right, align=left]},]

  \definecolor{orangy}{HTML}{F9A03F}
  \definecolor{greenish}{HTML}{3b754c}
  \tikzstyle{edgestyle_unnamed__1} = [thin,color=black]

  \definecolor{edgecolor_unnamed__2}{HTML}{001514}
  \tikzstyle{edgestyle_unnamed__2} = [thick,color=blackish]


  \foreach \i/\k in {1/0,2/0,2/1,4/3,5/3,5/4,7/6,8/6,8/7,10/9,11/9,11/10,13/12,14/12,14/13,16/15,17/15,17/16,19/18,20/18,20/19,22/21,23/21,23/22,25/24,26/24,26/25,28/27,29/27,29/28,31/30,32/30,32/31,34/33,35/33,35/34,37/36,38/36,38/37,40/39,41/39,41/40,43/42,44/42,44/43,46/45,47/45,47/46,49/48,50/48,50/49,52/51,53/51,53/52,55/54,56/54,56/55,58/57,59/57,59/58,61/60,62/60,62/61,64/63,65/63,65/64,67/66,68/66,68/67,70/69,71/69,71/70} {
   \draw[edgestyle_unnamed__2] (v\i_unnamed__1) -- (v\k_unnamed__1);
  }


  \foreach \i/\k in {5/2,6/3,8/1,9/4,10/7,11/0,17/14,18/15,20/13,21/16,22/19,23/12,29/26,30/27,32/25,33/28,34/31,35/24,36/0,36/11,37/12,37/23,38/24,38/35,39/3,39/6,40/15,40/18,41/24,41/35,41/38,42/3,42/6,42/39,43/27,43/30,44/12,44/23,44/37,45/15,45/18,45/40,46/27,46/30,46/43,47/0,47/11,47/36,48/1,48/8,49/13,49/20,50/25,50/32,51/4,51/9,52/16,52/21,53/25,53/32,53/50,54/4,54/9,54/51,55/28,55/33,56/13,56/20,56/49,57/16,57/21,57/52,58/28,58/33,58/55,59/1,59/8,59/48,60/7,60/10,61/14,61/17,62/26,62/29,63/19,63/22,64/2,64/5,65/26,65/29,65/62,66/31,66/34,67/2,67/5,67/64,68/14,68/17,68/61,69/7,69/10,69/60,70/19,70/22,70/63,71/31,71/34,71/66} {
   \draw[edgestyle_unnamed__1] (v\i_unnamed__1) -- (v\k_unnamed__1);
  }

  \foreach \i in {0,1,2,3,4,5,6,7,8,9,10,11,12,13,14,15,16,17,18,19,20,21,22,23,24,25,26,27,28,29,30,31,32,33,34,35,36,37,38,39,40,41,42,43,44,45,46,47,48,49,50,51,52,53,54,55,56,57,58,59,60,61,62,63,64,65,66,67,68,69,70,71} {
    \node at (v\i_unnamed__1) [vertexstyle_unnamed__1_0] {};
  }

\end{tikzpicture}
  \end{minipage}%
  \hspace{0.05\textwidth}
  \begin{minipage}{0.45\textwidth}
    \centering

\begin{tikzpicture}[x  = {(1cm,0cm)},
                    y  = {(0cm,1cm)},
                    z  = {(0cm,0cm)},
                    scale = 0.07,
                    color = {lightgray}]

  \coordinate (v0_unnamed__1) at (31.3775, 0);
  \coordinate (v1_unnamed__1) at (35.8775, -2.59808);
  \coordinate (v2_unnamed__1) at (35.8775, 2.59808);
  \coordinate (v3_unnamed__1) at (30.3083, 8.12109);
  \coordinate (v4_unnamed__1) at (35.3274, 6.77622);
  \coordinate (v5_unnamed__1) at (33.9825, 11.7953);
  \coordinate (v6_unnamed__1) at (27.1737, 15.6887);
  \coordinate (v7_unnamed__1) at (32.3698, 15.6887);
  \coordinate (v8_unnamed__1) at (29.7718, 20.1887);
  \coordinate (v9_unnamed__1) at (22.1872, 22.1872);
  \coordinate (v10_unnamed__1) at (27.2063, 23.5321);
  \coordinate (v11_unnamed__1) at (23.5321, 27.2063);
  \coordinate (v12_unnamed__1) at (15.6887, 27.1737);
  \coordinate (v13_unnamed__1) at (20.1887, 29.7718);
  \coordinate (v14_unnamed__1) at (15.6887, 32.3698);
  \coordinate (v15_unnamed__1) at (8.12109, 30.3083);
  \coordinate (v16_unnamed__1) at (11.7953, 33.9825);
  \coordinate (v17_unnamed__1) at (6.77622, 35.3274);
  \coordinate (v18_unnamed__1) at (0, 31.3775);
  \coordinate (v19_unnamed__1) at (2.59808, 35.8775);
  \coordinate (v20_unnamed__1) at (-2.59808, 35.8775);
  \coordinate (v21_unnamed__1) at (-8.12109, 30.3083);
  \coordinate (v22_unnamed__1) at (-6.77622, 35.3274);
  \coordinate (v23_unnamed__1) at (-11.7953, 33.9825);
  \coordinate (v24_unnamed__1) at (-15.6887, 27.1737);
  \coordinate (v25_unnamed__1) at (-15.6887, 32.3698);
  \coordinate (v26_unnamed__1) at (-20.1887, 29.7718);
  \coordinate (v27_unnamed__1) at (-22.1872, 22.1872);
  \coordinate (v28_unnamed__1) at (-23.5321, 27.2063);
  \coordinate (v29_unnamed__1) at (-27.2063, 23.5321);
  \coordinate (v30_unnamed__1) at (-27.1737, 15.6887);
  \coordinate (v31_unnamed__1) at (-29.7718, 20.1887);
  \coordinate (v32_unnamed__1) at (-32.3698, 15.6887);
  \coordinate (v33_unnamed__1) at (-30.3083, 8.12109);
  \coordinate (v34_unnamed__1) at (-33.9825, 11.7953);
  \coordinate (v35_unnamed__1) at (-35.3274, 6.77622);
  \coordinate (v36_unnamed__1) at (-31.3775, 0);
  \coordinate (v37_unnamed__1) at (-35.8775, 2.59808);
  \coordinate (v38_unnamed__1) at (-35.8775, -2.59808);
  \coordinate (v39_unnamed__1) at (-30.3083, -8.12109);
  \coordinate (v40_unnamed__1) at (-35.3274, -6.77622);
  \coordinate (v41_unnamed__1) at (-33.9825, -11.7953);
  \coordinate (v42_unnamed__1) at (-27.1737, -15.6887);
  \coordinate (v43_unnamed__1) at (-32.3698, -15.6887);
  \coordinate (v44_unnamed__1) at (-29.7718, -20.1887);
  \coordinate (v45_unnamed__1) at (-22.1872, -22.1872);
  \coordinate (v46_unnamed__1) at (-27.2063, -23.5321);
  \coordinate (v47_unnamed__1) at (-23.5321, -27.2063);
  \coordinate (v48_unnamed__1) at (-15.6887, -27.1737);
  \coordinate (v49_unnamed__1) at (-20.1887, -29.7718);
  \coordinate (v50_unnamed__1) at (-15.6887, -32.3698);
  \coordinate (v51_unnamed__1) at (-8.12109, -30.3083);
  \coordinate (v52_unnamed__1) at (-11.7953, -33.9825);
  \coordinate (v53_unnamed__1) at (-6.77622, -35.3274);
  \coordinate (v54_unnamed__1) at (0, -31.3775);
  \coordinate (v55_unnamed__1) at (-2.59808, -35.8775);
  \coordinate (v56_unnamed__1) at (2.59808, -35.8775);
  \coordinate (v57_unnamed__1) at (8.12109, -30.3083);
  \coordinate (v58_unnamed__1) at (6.77622, -35.3274);
  \coordinate (v59_unnamed__1) at (11.7953, -33.9825);
  \coordinate (v60_unnamed__1) at (15.6887, -27.1737);
  \coordinate (v61_unnamed__1) at (15.6887, -32.3698);
  \coordinate (v62_unnamed__1) at (20.1887, -29.7718);
  \coordinate (v63_unnamed__1) at (22.1872, -22.1872);
  \coordinate (v64_unnamed__1) at (23.5321, -27.2063);
  \coordinate (v65_unnamed__1) at (27.2063, -23.5321);
  \coordinate (v66_unnamed__1) at (27.1737, -15.6887);
  \coordinate (v67_unnamed__1) at (29.7718, -20.1887);
  \coordinate (v68_unnamed__1) at (32.3698, -15.6887);
  \coordinate (v69_unnamed__1) at (30.3083, -8.12109);
  \coordinate (v70_unnamed__1) at (33.9825, -11.7953);
  \coordinate (v71_unnamed__1) at (35.3274, -6.77622);

  \definecolor{blackish}{HTML}{001514}

  \tikzstyle{vertexstyle_unnamed__1_0} = [circle, scale=0.25, fill=blackish,label={[text=black, above right, align=left]},]

  \definecolor{orangy}{HTML}{F9A03F}
  \definecolor{greenish}{HTML}{3b754c}
  \tikzstyle{edgestyle_unnamed__1} = [thin,color=black]

  \definecolor{edgecolor_unnamed__2}{HTML}{001514}
  \tikzstyle{edgestyle_unnamed__2} = [thick,color=blackish]


  \foreach \i/\k in {5/2,6/3,8/1,9/4,10/7,11/0,17/14,18/15,20/13,21/16,22/19,23/12,29/26,30/27,32/25,33/28,34/31,35/24,36/0,36/11,37/12,37/23,38/24,38/35,39/3,39/6,40/15,40/18,41/24,41/35,41/38,42/3,42/6,42/39,43/27,43/30,44/12,44/23,44/37,45/15,45/18,45/40,46/27,46/30,46/43,47/0,47/11,47/36,48/1,48/8,49/13,49/20,50/25,50/32,51/4,51/9,52/16,52/21,53/25,53/32,53/50,54/4,54/9,54/51,55/28,55/33,56/13,56/20,56/49,57/16,57/21,57/52,58/28,58/33,58/55,59/1,59/8,59/48,60/2,60/5,61/14,61/17,62/26,62/29,63/7,63/10,64/19,64/22,65/26,65/29,65/62,66/7,66/10,66/63,67/31,67/34,68/14,68/17,68/61,69/19,69/22,69/64,70/31,70/34,70/67,71/2,71/5,71/60} {
   \draw[edgestyle_unnamed__1] (v\i_unnamed__1) -- (v\k_unnamed__1);
  }


  \foreach \i/\k in {1/0,2/0,2/1,4/3,5/3,5/4,7/6,8/6,8/7,10/9,11/9,11/10,13/12,14/12,14/13,16/15,17/15,17/16,19/18,20/18,20/19,22/21,23/21,23/22,25/24,26/24,26/25,28/27,29/27,29/28,31/30,32/30,32/31,34/33,35/33,35/34,37/36,38/36,38/37,40/39,41/39,41/40,43/42,44/42,44/43,46/45,47/45,47/46,49/48,50/48,50/49,52/51,53/51,53/52,55/54,56/54,56/55,58/57,59/57,59/58,61/60,62/60,62/61,64/63,65/63,65/64,67/66,68/66,68/67,70/69,71/69,71/70} {
   \draw[edgestyle_unnamed__2] (v\i_unnamed__1) -- (v\k_unnamed__1);
  }

  \foreach \i in {0,1,2,3,4,5,6,7,8,9,10,11,12,13,14,15,16,17,18,19,20,21,22,23,24,25,26,27,28,29,30,31,32,33,34,35,36,37,38,39,40,41,42,43,44,45,46,47,48,49,50,51,52,53,54,55,56,57,58,59,60,61,62,63,64,65,66,67,68,69,70,71} {
    \node at (v\i_unnamed__1) [vertexstyle_unnamed__1_0] {};
  }

\end{tikzpicture}
  \end{minipage}
  \caption{The relation graphs for $P$ and $Q$, the two quantum isomorphic but not isomorphic matroids.} \label{fig:matroids-PQ}
\end{figure}
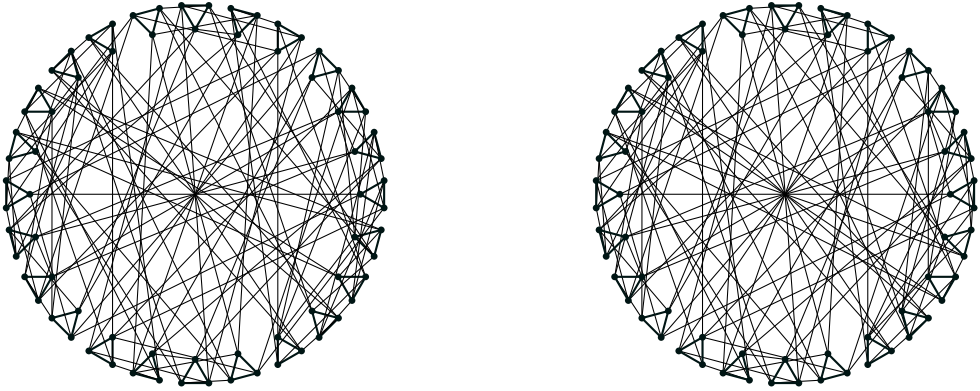
  
  \begin{theorem}
  \label{thm:noniso-quantum-ism}
      The matroids $P$ and $Q$ are $\NB_{\bullet}$--quantum isomorphic but not isomorphic.
  \end{theorem}

Before proving this theorem, we point out the numerous similarities between the matroids $P$ and $Q$ despite being nonisomorphic. These two matroids have the same rank and ground-set size (cf. \Cref{thm:gs+same+size}). They have the same number of bases, independent sets, circuits, flats, and hyperplanes. They both have connectivity equal to $3$. They have the same Tutte polynomial, and therefore the same characteristic polynomial. Their automorphism groups are both isomorphic to a semidirect product $(S_{4} \times S_{4}) \rtimes (\Z/2\Z)$.  
These may be verified using any software that has a matroids package, we used OSCAR \cite{OSCAR-book, Oscar}. Because these properties are not critical to our main results, we omit the proofs. 
  
\begin{proof}[Proof of \Cref{thm:noniso-quantum-ism}]
  The matroids $P$ and $Q$ are $\NB_{\bullet}$--quantum isomorphic by \Cref{thm:qisomorphic-matroid} and the fact that the LBCS used to define $Q$ has a perfect quantum strategy. 

Next, we show that $P$ and $Q$ are not isomorphic by exhibiting a matroid that is a minor of one but not the other. 
    Let $N$ be the rank $3$ matroid on $\{1,\ldots,9\}$ whose nonbases are
    \begin{equation*}
        \NB(N) = \{123, 456, 789\}. 
    \end{equation*}
    We claim that there is a restriction of $Q$ that is isomorphic to  $N$ but no restriction of $P$ is isomorphic to $N$. 
    Let $Y\subset E(Q)$ be the subset
    \begin{equation*}
        Y = \{(1,+1), (2,+1), (3,+1), (4,+1), (5,+1), (6,+1), (7,-1), (8,-1), (9,-1)\}
    \end{equation*}
    The matroid $Q|Y$ is isomorphic to $N$ as the $3$ nonbases of $Q|Y$ are $\{(1,+1), (2,+1), (3,+1)\}$, $\{(4,+1), (5,+1), (6,+1)\}$, and $\{(7,-1), (8,-1), (9,-1)\}$. 

    Suppose by way of contradiction that $X$ is a $9$-element subset of $E(P)$ so that $P|X$ is isomorphic to $N$; label the nonbases of this matroid by $H_1$, $H_2$, and $H_3$. Consider the case where, e.g., $(1,+1)$ and $(1,-1)$ are both in $X$, say $(1,+1) \in H_1$ and $(1,-1) \in H_2$. Note that $\pi(H_1)$ cannot equal $\pi(H_2)$. Assume without loss of generality that $\pi(H_1)=123$ and $\pi(H_2) = 147$.  The set $\pi(H_3)$ meets $123$ or $147$ and does not contain $1$. Without loss of generality, suppose $2\in \pi(H_3)$, so $\pi(H_3) = 258$. This means that both $(2,+1)$ and $(2,-1)$ are in $X$. If $H_1 = \{(1,+1), (2,+1), (3,+1)\}$, then $\{(1,-1), (2,-1), (3,+1)\}$ is a 4th nonbasis, and if  $H_1 = \{(1,+1), (2,-1), (3,-1)\}$, then $\{(1,-1), (2,-1), (3,+1)\}$ is a 4th nonbasis. In any case, $X$ cannot have both $(a,+1)$ and $(a, -1)$ for any $a \in \{1,\ldots,9\}$. 

    By the previous paragraph, we conclude that $\pi(X) = \{1,\ldots, 9\}$. Because of this and the fact that the nonbases of $P|X$ are pairwise disjoint, either $\{\pi(H_1), \pi(H_2), \pi(H_3)\}$ equals $\{123, 345, 789\}$ or $\{147,258,369\}$; without loss of generality suppose $\pi(H_1)=123$, $\pi(H_2)=258$, and $\pi(H_3)=789$.   The assignments of values $x_1, \ldots, x_9$ must satisfy 
    \begin{equation*}
        x_1 x_2 x_3 = x_4 x_5 x_6 = x_7 x_8 x_9 = 1 \quad \text{and} \quad x_1 x_4 x_7 = x_2 x_5 x_8 = x_3 x_6 x_9 = -1.
    \end{equation*}
    In the first case, the product $x_1 \cdots x_9$ equals $1$ but in the second case $x_1 \cdots x_9$ equals  $-1$, which is not possible.  So there is no restriction of $P$ that is isomorphic to $N$, as required. 
  \end{proof}

\section{The isomorphism algebra}
\label{sec:algebra}
In this section we give a purely algebraic characterization of when two matroids are $\pS$--quantum isomorphic. We do this by defining a $\ast$-algebra $\QMat M N \pS$ which is nonzero if and only if the $(M,N,\pS)$--isomorphism game has a perfect quantum commuting strategy exists, cf.   \Cref{thm:algebraic-quantum-iso}. We also gather some basic matroid properties that are shared by $\pS$--quantum isomorphic matroids. 

\subsection{Isomorphism algebras} 
Let $E$ and $F$ be finite sets. Define the noncommutative polynomial ring in $|E|\times |F|$ variables
\begin{equation}
\label{eq:CEF}
 \C \langle E, F \rangle = \C\langle u_{ax} \suchthat a\in E, x\in F \rangle.
\end{equation}
Equip this ring with a $\ast$-involution induced by $u_{ax}^{\ast} = u_{ax}$. The \textit{quantum isomorphism $\ast$-algebra} is 
\begin{equation}
\label{eq:fSEF}
    \fS(E,F) = \C\langle E, F\rangle / I
\end{equation}
    where $I\subset \C\langle E, F\rangle$ is the self-adjoint (two-sided) ideal generated by the forms
\begin{gather*}
u_{ax}^2 - u_{ax}, \quad 
u_{ax}u_{ay} \ (x\neq y), \quad 
u_{ax}u_{bx} \ (a\neq b), \quad 
1-\sum_{a\in E} u_{ax}, \quad  
1-\sum_{x\in F} u_{ax}.
\end{gather*}
The matrix of variables $[u_{ax}]$ satisfying the above relations is called a \textit{magic unitary}. 
Observe that if $E$ and $F$ have different sizes, then $\fS(E,F)$ is 0, so all nonzero magic unitaries must be square. This fact is well established, see, e.g., footnote 4 in \cite{Weber}. 

When $E=F$, we have that $\fS(E,F)$ is naturally identified with the quantum permutation algebra, a $\ast$-algebra version of Wang's quantum symmetric group \cite{Wang:1999}.  
By a straightforward generalization of the coproduct on the quantum permutation algebra, we may define the \textit{cocomposition}, i.e., a unital $\ast$-homomorphism defined by
\begin{equation}
\label{eq:cocomp-sym}
    \Delta \colon  \fS(D,F) \to  \fS(D,E) \otimes  \fS(E,F),
    \quad u_{ax} \mapsto \sum_{s\in E} u_{as} \otimes u_{sx}.  
\end{equation}

\subsection{Isomorphism algebra associated to the game}
\label{sec:iso-algebra-game}

The isomorphism algebra associated to the $(M,N,\pS)$--isomorphism game is
\begin{equation*}
   \QMat{M}{N}{\pS} = \fS(\pssets{\pS}{M},\pssets{\pS}{N}) / J_{\bullet}(M,N,\pS)
\end{equation*}
where $J_{\bullet}(M,N,\pS)$ is the self-adjoint, two-sided ideal 
\begin{equation*}
    J_{\bullet}(M,N,\pS)
    = \langle u_{\fa \fx} u_{\fb \fy} \colon 
    \rel_{\pS(M)}(\fa, \fb) \neq \rel_{\pS(N)}(\fx, \fy) \rangle.
\end{equation*}

The cocompositions on the algebras $\fS(\pssets{\pS}{M},\pssets{\pS}{N})$ define cocompositions on the isomorphism algebras $\QMat{M}{N}{\pS}$ as we see in the next proposition. 
\begin{proposition}
    The map in \Cref{eq:cocomp-sym} induces a  unital $\ast$-homomorphism
\begin{equation}
\label{eq:cocomp}
    \Delta \colon \QMat{L}{N}{\pS} \to \QMat{L}{M}{\pS} \otimes \QMat{M}{N}{\pS} 
\end{equation}
\end{proposition}

\begin{proof}
Consider 
\begin{equation*}
     \Delta(u_{\fa\fx} u_{\fb\fy}) = \sum_{\fs, \ft} u_{\fa\fs} u_{\fb\ft} \otimes u_{\fs \fx} u_{\ft \fy} 
\end{equation*}
If $\rel_{\pS(L)}(\fa, \fb) \neq \rel_{\pS(N)}(\fx, \fy)$ then each term on the right is zero $\QMat{L}{M}{\pS} \otimes \QMat{M}{N}{\pS}$ because each $\rel_{\pS(M)}(\fs, \ft)$ cannot simultaneously equal $\rel_{\pS(L)}(\fa, \fb)$ and $\rel_{\pS(N)}(\fx, \fy)$. 
\end{proof}

For $M=N$, we write $\QAlg{M}{\pS} = \QMat{M}{M}{\pS}$ and the cocomposition in Formula \eqref{eq:cocomp} defines a coproduct on $\QAlg{M}{\pS}$. 
We define the $\pS$-\textit{quantum automorphism group} to be $\QAutp{M}{\pS} =(\QAlg{M}{\pS}, \Delta)$.  

Let us explain how $\QAutp{M}{\pS}$ qualifies as a quantization of the ordinary automorphism group of $M$. Denote by $\Sym(E)$ the group of permutations of the set $E$.  Recall from \S\ref{sec:matroid-iso-str} that $\Aut(M)$ is defined as a subgroup of $\Sym(\ssets E M)$. The group $\Aut(M)$ acts on $\psSets M$ by $\sigma(\fa) = (\sigma(S_{\fa}), \sigma(p_{\fa}))$ for $\sigma \in \Aut(M)$ and $\fa\in \psSets M$. This action is faithful and so it affords an injective group homomorphism $\Aut(M) \to \Sym(\psSets M)$; denote the image by $\Aut_{\bullet}(M,\pS)$. Set
\begin{equation*}
    \QAlg{M}{\pS}^{\com} = \QAlg{M}{\pS} / \langle u_{\fa\fx} u_{\fb\fy} - u_{\fb\fy} u_{\fa\fx}  \suchthat \fa,\fb,\fx,\fy\in \psSets M \rangle.
\end{equation*}
Given a permutation group $G\subset \Sym(\psSets M)$, one obtains a quantum group $C(G)$ in a standard way. 
The way we view $\QAutp{M}{\pS}$ as a quantization of $\Aut_{\bullet}(M,\pS)$ is given by the following proposition. 
\begin{proposition}
\label{prop:quantization}
    We have an identification 
    \begin{equation*}
        \QAlg{M}{\pS}^{\com} = C(\Aut_{\bullet}(M,\pS)).
    \end{equation*}
\end{proposition}

\begin{proof}
This is a straightforward application of \cite[Exer.~1.10]{Freslon} and \Cref{cor:induced+iso}.
\end{proof}

The significance of the isomorphism algebra $\QMat{M}{N}{\pS}$ is that it gives an algebraic characterization of when the $(M,N,\pS)$--isomorphism game has a perfect quantum strategy. Since the $(M,N,\pS)$--isomorphism game is a colored graph isomorphism game, and the result is known in the uncolored case (see \cite{BrannanChirvasituEiflerHarrisPaulsenSuWasilewski}), we will only give a sketch of the proof.
\begin{theorem}
\label{thm:algebraic-quantum-iso}
Let $\pS$ be a matroid isomorphism structure that covers the matroids $M$ and $N$. The $(M,N,\pS)$--isomorphism game has a perfect quantum commuting strategy if and only if $M$ and $N$ are $\pS$-quantum isomorphic, i.e., $\QMat{M}{N}{\pS}$ is nonzero.
\end{theorem}

\begin{proof}
    Suppose that the $(M,N,\pS)$--isomorphism game has a perfect quantum commuting strategy. By \Cref{thm:perfectstrat}, there exists a nonzero representation of $\QMat{M}{N}{\pS}$, which implies that $\QMat{M}{N}{\pS}$ is nonzero.

    Conversely, suppose $\QMat{M}{N}{\pS}$ is nonzero. Let $\fA =\QAutp{M}{\pS}$ and $\fB=\QAutp{N}{\pS}$.  We will show that $\QMat{M}{N}{\pS}$ is a $(\fA,\fB)$--bigalois extension and use results of \cite{BrannanChirvasituEiflerHarrisPaulsenSuWasilewski} to obtain a tracial state on $\QMat{M}{N}{\pS}$. The existence of a perfect quantum commuting strategy then follows from \Cref{thm:perfectstrat}.

    The cocomposition in Formula \eqref{eq:cocomp} yields the following left, respectively right  comodule actions:
    \begin{align*}
        \alpha: \QMat{M}{N}{\pS}&\to \QAlg{M}{\pS} \otimes \QMat{M}{N}{\pS},\quad u_{\fa \fb}\mapsto \sum_\fk u_{\fa \fk} \otimes u_{\fk \fb},\\
        \beta: \QMat{M}{N}{\pS}&\to  \QMat{M}{N}{\pS}\otimes \QAlg{N}{\pS}, \quad u_{\fa \fb}\mapsto \sum_\fk u_{\fa \fk} \otimes u_{\fk \fb}.
    \end{align*}
    Furthermore, the linear map
    \begin{align*}
        k_l: \QMat{M}{N}{\pS}\otimes \QMat{M}{N}{\pS} \to \QAlg{M}{\pS} \otimes \QMat{M}{N}{\pS},\quad x\otimes y \mapsto \alpha(x)(1\otimes y)
    \end{align*}
    is bijective, with inverse $\eta_l=(\mathrm{id}\otimes m)(\gamma\otimes \mathrm{id})$. Here, $m(x \otimes y)=xy$ and $\gamma(u_{\fa \fb})=\sum_\fk u_{\fa \fk} \otimes u_{\fb \fk}$. This shows that $(\QMat{M}{N}{S}, \alpha)$ is a left $\fA$-Galois extension. Similarly, one gets that $(\QMat M N S, \beta)$ is a right $\fB$-Galois extension. One may verify that $(\mathrm{id}\otimes \beta)\circ \alpha=(\alpha\otimes\mathrm{id})\circ \beta$. In summary,  $\QMat{M}{N}{S}$ is a $(\fA,\fB)$--bigalois extension.
    Using Theorems 3.15 and 3.17 of \cite{BrannanChirvasituEiflerHarrisPaulsenSuWasilewski}, we have a faithful tracial state $\tau \colon \QMat{M}{N}{\pS}\to \C$.  Now, \Cref{thm:perfectstrat} guarantees the existence of a perfect quantum commuting strategy.
\end{proof}

\subsection{Further properties of quantum isomorphic matroids} If $M$ and $N$ are $\pS$-quantum isomorphic, then $|\psSets M| = |\psSets N|$ because the magic unitary $[u_{\fa\fx}]$ must be square.  We record other properties that must be preserved by quantum isomorphism. 

\begin{lemma}
\label{lem:same-size}
     Suppose $M$ and $N$ are $\pS$-quantum isomorphic. If $\fa \in \pssets \pS M$, $\fx \in \pssets \pS N$ satisfy $|S_{\fa}| \neq |S_{\fx}|$, then  
     \begin{align*}
         u_{\fa \fx} \equiv 0 \quad \text{in} \quad \QMat M N \pS.
     \end{align*}
\end{lemma}

\begin{proof}
  Let $\fa \in \psSets M$, $\fx \in \psSets N$, and set $A = S_{\fa}$, $X = S_{\fx}$.  Then
\begin{align*}
    u_{\fa \fx} \cdot |\psSets M| 
    = u_{\fa\fx} \cdot \sum_{\ft \in \psSets N} \sum_{\fs \in \psSets M}u_{\fs\ft}
  = \sum_{\ft \in \psSets N} \sum_{\substack{\fs\in \psSets M\\ S_{\fs} \neq A}} u_{\fa \fx}u_{\fs \ft} 
  + \sum_{\ft \in \psSets N} \sum_{\substack{\fs \in \psSets M\\ S_{\fs} = A}} u_{\fa\fx}u_{\fs\ft}
\end{align*}
In the double sum on the left, we may use a row sum relation of $\fS(\pssets{\pS}{M},\pssets{\pS}{N})$ to get
\begin{equation*}
    \sum_{\ft \in \psSets N} \sum_{\substack{\fs\in \psSets M\\ S_{\fs} \neq A}} u_{\fa \fx}u_{\fs \ft} = \sum_{\substack{\fs  \in \psSets M\\ S_{\fs} \neq A}}  u_{\fa\fx}  =  u_{\fa\fx} \cdot ( |\psSets M| - |A|).
\end{equation*}
In the double sum on the right, we may use the relation $u_{\fa\fx}u_{\fs\ft} = 0$ for $\rel_{\psSets M}(\fa, \fs) \neq \rel_{\psSets N}(\fx,\ft)$ to get
\begin{equation*}
    \sum_{\ft \in \psSets N} \sum_{\substack{\fs \in \psSets M\\ S_{\fs} = A}} u_{\fa\fx}u_{\fs\ft} = 
    \sum_{\substack{\ft \in \psSets N \\ S_{\ft} = X}} \sum_{\substack{\fs \in \psSets M\\ S_{\fs} = A}} u_{\fa\fx}u_{\fs\ft} = \sum_{\substack{\ft \in \psSets N \\ S_{\ft} = X}} \sum_{\fs \in \psSets M} u_{\fa\fx}u_{\fs\ft} = u_{\fa\fx} \cdot |X|.
\end{equation*}
Combining these two calculations, we see that 
\begin{equation*}
    u_{\fa \fx} \cdot |\psSets M|  = u_{\fa\fx} \cdot(|\psSets M| - |A| + |X|),
\end{equation*}
which implies that $|A| = |X|$ or $u_{\fa\fx} = 0$, as required.
\end{proof}

Denote by $\pS(M, r)$ the elements of $\pS(M)$ that have exactly $r$ elements. 
\begin{proposition}
\label{prop:misc+properties}
    Suppose $M$ and $N$ are $\pS$-quantum isomorphic. 
    \begin{enumerate}
        \item We have $|\pS(M,r)| = |\pS(N,r)|$ for each $r\geq 1$. In particular,  $|\ssets \pS M| = |\ssets \pS N|$. 
        \item If $\pS$ is $\pB$ or $\NB$ then $M$ and $N$ have the same rank and their ground sets have the same size. 
        \item If $\pS = \pC$ and $\rk(M) = \rk(N)$ then $M$ is paving if and only if $N$ is paving.
    \end{enumerate}
\end{proposition}

\begin{proof}
Consider (1). We have
    \begin{equation*}
        \sum_{\substack{\fb \in \psSets M \\ |S_{\fb}| = |S_{\fx}| }} u_{\fb\fx} = \sum_{\fb \in \psSets M} u_{\fb\fx} = 1 = \sum_{\fy \in \psSets N} u_{\fa\fy} = \sum_{\substack{\fy \in \psSets N \\ |S_{\fy}| = |S_{\fa}| }} u_{\fa\fy}.
    \end{equation*}
So, for a fixed $r$, the submatrix $[u_{\fa\fx} \suchthat |S_{\fa}| = |S_{\fx}| = r]$ is a magic unitary, and hence square. Because we have, e.g., that the number of elements $\fa$ of $\psSets M$ that have $r$ elements is $r |\pS(M,r)|$, we get the desired equality $|\pS(M,r)| = |\pS(N,r)|$. 
  
For (2), suppose  $\rk(M) \neq \rk(N)$. This implies that $|A| \neq |X|$ for all $\fa = (A,p) \in \pssets \pB M$ and $\fx = (X,q) \in \pssets \pB N$. This in turn causes the isomorphism algebra to be zero, since $u_{\fa\fx} = 0$ for all such $\fa$ and $\fx$ by \Cref{lem:same-size}.

 Finally assume $M$ is paving and let $\rk(M)=r= \rk(N)$. By $(1)$ we have
  \[
    |\pC(M)| = |\pC(M,r)| + |\pC(M,r+1)| = \sum_{k \geq 1} |\pC(N,k)| = |\pC(N)|.
  \]
  Knowing that $|\pC(M,k)| = |\pC(N,k)|$ for all $k\geq0$ we find that $N$ has only circuits of size $r$ and $r+1$, resulting in $N$ being paving. The converse is similar.
\end{proof}

\section{Isomorphism algebras on the ground set}
The variables appearing in the quantum automorphism groups $\QAutp{M}{\pS}$ are indexed by pairs of elements in $\psSets M$. The matroid quantum automorphism groups appearing in \cite{CoreyJoswigSchanzWackWeber} are defined as compact matrix quantum groups whose variables are indexed by the ground set. We extend the definitions in \cite{CoreyJoswigSchanzWackWeber} to isomorphism algebras $\GS M N \pS$ and compare these to the algebras $\QMat M N \pS$. 

\subsection{Ground set quantum isomorphism of matroids}
Let $M$ be a matroid of rank $r$. Recall from \cite{CoreyJoswigSchanzWackWeber} that a tuple $A\in \ssets{E}{M}^{s}$ is an \textit{independent tuple} (respectively, a \textit{basis tuple} or \textit{hyperplane tuple}) if $A$ has no repeated elements and its underlying set in an independent set (respectively, a basis or a hyperplane) of $M$. Furthermore, $A$ is a \textit{circuit tuple} if $A = (a,a)$ for some nonloop element $a\in A$ or if $A$ has no repeated elements and its underlying set is a circuit of $M$. A \textit{nonbasis tuple} is a length-$r$ tuple that is not a basis tuple.  

Let $\pS$ be one of the matroid isomorphism structures: independent sets, bases, nonbases, circuits, or hyperplanes. Denote by $\overline{\pS}(M)$ the set of $\pS$-tuples of $M$ as defined above.
The algebra $\GS M N \pS$ we define below is a quotient of the noncommutative polynomial ring $\C\langle \ssets E M, \ssets E N \rangle$, but we denote the variables by $w_{ax}$ to distinguish them from the variables we use in $\QMat M N \pS$.  
Given tuples $A = (a_1,\ldots,a_s)$ and $X = (x_1,\ldots,x_s)$, define
\begin{equation*}
    w_{AX} = w_{a_1x_1} \cdots w_{a_sx_s}.
\end{equation*}
The \textit{$\pS$--isomorphism algebra} of matroids $M$ and $N$ is
\begin{equation*}
    \GS M N \pS = \fS(\ssets E M , \ssets E N ) / K(M, N, \pS)
\end{equation*}
where $K(M, N, \pS)$ is the self-adjoint, two-sided ideal of $\fS(\ssets E M , \ssets E N )$ defined by
\begin{equation*}
    K(M, N, \pS) = 
    \langle 
    w_{AX} \suchthat (A\in \overline{\pS}(M) \text{ and } X\notin \overline{\pS}(N)) \text{ or }
    (A\notin \overline{\pS}(M) \text{ and } X\in \overline{\pS}(N)) 
    \rangle.
\end{equation*}
Note that $\GS M N \pB =  \GS M N \NB$.
Given matroids $L$, $M$ and $N$ we have a cocomposition
\begin{align*}
  \GS L N \pS &\to \GS L M \pS \otimes \GS M N \pS \\
    \quad w_{ax} &\mapsto \sum_{p\in \ssets{E}{M}} w_{ap} \otimes w_{px}.
\end{align*}
When $M$ and $N$ are equal, the algebras $\GS M N \pS$ recover the quantum automorphism groups $\QAut{\pS}{M}$ of $M$ in \cite{CoreyJoswigSchanzWackWeber}. 

\subsection{Ground sets of quantum isomorphic matroids}
We proved in \Cref{prop:misc+properties} that some basic numerical properties of matroids must be preserved under quantum isomorphism. We now  prove that the ground sets of quantum isomorphic matroids must have the same size by comparing magic unitaries on the ground set to those on pointed sets. 

Suppose that $\pS$ covers the matroids $M$ and $N$, and for each $a\in \ssets{E}{M}$ fix $\fb_{a} \in \pssets{\pS}{M}$ with $p_{\fb_{a}} = a$.  Define 
\begin{equation}
\label{eq:gs-to-pointed}
    \varphi \colon 
    \C\langle E(M), E(N)\rangle 
    \to 
    \C \langle \pS_{\bullet}(M), \pS_{\bullet}(N)  \rangle,  
    \quad  
    w_{ax} \mapsto \sum_{\substack{\ft \in \psSets N \\ p_{\ft} = x}} u_{\fb_a\ft}.
\end{equation}

\begin{proposition}
\label{prop:sym-comparison}
    The map $\varphi$ is independent of the $\fb_{a}$ and induces a morphism
    \begin{equation*}
        \varphi \colon \fS(\ssets E M, \ssets E N) \to \fS(\psSets M, \psSets N).
    \end{equation*}
\end{proposition}

\noindent First consider the following lemma.

\begin{lemma}
\label{lem:comparison-independent}
Fix $\fb = (B, a) \in \pssets{\pS}{M}$ and $\fy = (Y, x) \in \pssets{\pS}{N}$. Then
\begin{align*}
    \sum_{\ft: p_{\ft} = x} u_{\fb\ft} = \sum_{\fs: p_{\fs} = a}u_{\fs\fy} \quad \text{in} \quad \fS(\psSets M, \psSets N). 
\end{align*}
\end{lemma}

\begin{proof}
 Throughout, all equations take place in $\fS(\psSets M, \psSets N)$. Using 
\begin{equation*}
  \sum_{\fc\in \pssets{\pS}{M}} u_{\fc\fy}=1 
  \quad \text{and} \quad 
  u_{\fb\ft}u_{\fc\fy} = 0 \quad \text{whenever} \quad a\neq p_{\fc},
\end{equation*}
we have
\begin{align*}
\sum_{\ft: p_{\ft} = x} u_{\fb\ft}
= \sum_{\ft: p_{\ft} = x} u_{\fb\ft}
\left( 
\sum_{\fc \in \pssets{\pS}{M}} u_{\fc\fy}
\right)
=\sum_{\ft:p_{\ft}=x}\sum_{\fc:p_{\fc}=a} 
u_{\fb\ft}u_{\fc\fy}.
\end{align*}

\noindent Now, using the equation 
\begin{equation*}
    \sum_{\fz \in \pssets{\pS}{N}}u_{\fb\fz} = 1 
    \quad \text{and} \quad 
    u_{\fb\fz}u_{\fc\fy}=0 \quad \text{whenever} \quad  x\neq p_{\fz},
\end{equation*}
we infer
\begin{equation*}
    \sum_{\fs: p_{\fs} = a}u_{\fs\fy} 
    = \left(\sum_{\fz \in \pssets{\pS}{N}}  
    u_{\fb\fz} \right) \sum_{\fs: p_{\fs} = a}u_{\fs\fy} 
    = \sum_{\fz:p_{\fz} = x} \sum_{\fs: p_{\fs} = a} u_{\fb\fz} u_{\fs\fy} 
\end{equation*}
and the result follows. 
\end{proof}

\begin{proof}[Proof of \Cref{prop:sym-comparison}]
We first consider the independence on the $\fb_a$. By Lemma \ref{lem:comparison-independent}, the map $\varphi$ may be defined as 
\begin{equation}
\label{eq:comparison-N}
    \varphi(w_{ax}) = \sum_{\substack{\fs\in \psSets{N} \\ p_{\fs}=a}} u_{\fs\fy_{x}}
\end{equation}
where we fix $\fy_{x}\in \psSets{N}$ for each $x\in \ssets E N$. In particular, the original characterization of $\varphi$ is independent of the $\fb_a$ and the characterization is \Cref{eq:comparison-N} is independent of the $\fy_{x}$. 

Consider the idempotent relation
\begin{equation*}
    \varphi(w_{ax}^2) = 
    \sum_{\substack{\fs,\ft\in \psSets N \in \ssets \pS N \\ x=p_{\fs}=p_{\ft}}} 
    u_{\fb_a\fs} u_{\fb_a\ft} 
\end{equation*}
Since $u_{\fb_a\fs} u_{\fb_a\ft} = 0$ for  $\fs\neq \ft$ and $(u_{\fb_a\fs})^2 = u_{\fb_a\fs}$, the above sum is just $\varphi(w_{ax})$. By a similar argument, we get that $ \varphi(w_{ax}w_{ay}) = 0$ for $x\neq y$. If  $a\neq b$, then $u_{\fb_a\fs} u_{\fb_{b}\ft} = 0$ since 
\begin{equation*}
    \rel_{\pS(M)}(\fb_a,\fb_b) \neq \rel_{\pS(N)}(\fs,\ft) 
\end{equation*} 
when $p_{\fs} =p_{\ft} = x$.
So we have that $\varphi(w_{ax}w_{bx})= 0$. 
Next, we have
\begin{equation*}
    \varphi\left( \sum_{x\in E(N)} w_{ax} \right) = \sum_{x\in E(N)} 
    \sum_{\substack{\fs\in \psSets{N} \\ p_{\fs} = x}} u_{\fb_{a} \fs}  = \sum_{\fs\in \psSets{N}} u_{\fb_{a}\fs} = 1. 
\end{equation*}
In a similar way, using Equation \eqref{eq:comparison-N}, we also have 
\begin{equation*}
    \varphi\left( \sum_{a\in E(M)} w_{ax} \right) = 1. 
\end{equation*}
as required. 
\end{proof}

\begin{theorem}
\label{thm:gs+same+size}
    If $M$ and $N$ are $\pS$-quantum isomorphic, then $E(M)$ and $E(N)$ have the same size. 
\end{theorem}

\begin{proof}
    By \Cref{prop:sym-comparison}, $[\varphi(w_{ax})]$ defines a nonzero $|E(M)| \times |E(N)|$ magic unitary, so $|E(M)| = |E(N)|$ because nonzero magic unitaries must be square.
\end{proof}

\subsection{A disjoint automorphism criterion}\label{sec:disj-aut}
Recall from \Cref{def:relation+colored+graph} that we may associate a colored graph $\Grel(M,\pS)$ to a matroid $M$ and isomorphism structure $\pS$. 
In this section, we show that $\QAutp M \pS$ may be naturally identified with the quantum automorphism group of $\Grel(M,\pS)$ in the sense of \cite{robersonschmidt}. From this, we get a disjoint automorphism criterion for $\QAutp M \pS$ which ensures that $\QAlg{M}{\pS}^{\com} \neq \QAlg M \pS$. In view of \Cref{prop:quantization}, those $\QAlg M \pS$ that are noncommutative are proper quantizations of $\Aut_{\bullet}(M,\pS)$.

A matroid isomorphism $\varphi \colon M \to N$ defines a colored graph isomorphism $\varphi_{\Gamma} \colon \Grel(M,\pS) \to \Grel(N,\pS)$ by $\varphi_{\Gamma}(\fa) = (\varphi(S_{\fa}), \varphi(p_{\fa}))$. The following is a reformulation of \Cref{cor:induced+iso}. 

\begin{proposition}
    If $\Theta \colon \Grel(M,\pS) \to \Grel(N,\pS)$ is a colored graph isomorphism, then there is an isomorphism $\varphi \colon M \to N$ such that the induced colored graph isomorphism $\varphi_{\Gamma}$ equals $\Theta$. 
\end{proposition}

Let $G = G(M,\pS)$. We denote by $A_{\Grel^c}$ the adjacency matrix of the edge color $c$, that is, $(A_{G^c})_{\fa \fx}=1$ if $\rel_{\ssets \pS M}(\fa , \fx) = c$ and $(A_{G^c})_{\fa\fx}=0$ otherwise. This lets us relate the quantum automorphism group of $M$ to that of $G$.

\begin{proposition}
Let $M$ be a matroid, $\pS$ a matroid isomorphism structure that covers $M$, and $G = \Grel(M,\pS)$. Then we have the following presentation of $\QAlg M \pS$
\begin{align*}
    \QAlg M \pS = \fS(\ssets E M , \ssets E M) / \langle A_{G^c}U=UA_{G^c} \text{ for } c\in \{1,2\} \rangle
\end{align*}
\end{proposition}

\begin{proof}
This follows (for example) from \cite[Prop.~2.1.6]{Schmidt-thesis}, where it is shown that $A_{\Grel^c}U=UA_{\Grel^c}$ is equivalent to $u_{\fa\fx}u_{\fb\fy}=0$ for $\rel_{\ssets \pS M}(\fa , \fb)=c, \rel_{\ssets \pS M}(\fx , \fy)\neq c$.
\end{proof}

A pair of permutations $\pi, \tau \in \Sym(E)$ (where $E$ is a finite set) are \textit{disjoint} if $\pi(i)\neq i$ implies $\tau(i)=i$ and vice versa. 
In our case, we study the disjoint automorphism of the relation graph. The proof is the same as in \cite{Foldedcube}. 

\begin{theorem}
\label{thm:disjoint+aut+criterion}
  Let $M$ be a matroid and $\pS$ be an isomorphism structure that covers $M$. If $\Grel(M,\pS)$ has a pair of non-trivial, disjoint automorphism $\tilde{\pi}, \tilde{\tau}$, then the algebra $\QAlg{M}{\pS}$ is noncommutative.
  \label{lem:disjoint_aut}
\end{theorem}

Using this theorem, we produce examples of matroids such that the underlying $\ast$-algebra of $\QAutp M \NB$ is noncommutative, but the (bases) quantum automorphism group from \cite{CoreyJoswigSchanzWackWeber} is commutative. The following proposition gives a straightforward application of \Cref{thm:disjoint+aut+criterion} for finding $\QAutp N \pS$ that exhibit quantum symmetries. 

\begin{proposition}
\label{prop:examples+noncommutative}
    Let $M$ be a matroid and $\pS$ a matroid isomorphism structure covering $M$. 
    Suppose there are distinct elements $a,b,c,d\in \ssets E M$ such that 
    \begin{itemize}
        \item there are $A, B\in \ssets \pS M$ with $a,b\in A$ and $c,d\in B$, and
        \item the elements $a,b,c,d$ are contained in no other elements of $\pS$.
    \end{itemize}
    Then the transpositions $a \leftrightarrow b$ and $c\leftrightarrow d$ are a pair of disjoint automorphisms of $\Grel(M,\pS)$.  In particular, $\QAlg M \pS$ is noncommutative. 
\end{proposition}

Theorem~5.3 in \cite{CoreyJoswigSchanzWackWeber} asserts that $\QAut{\pB}{M}$ (which is the same as $\QAut{\NB}{M}$) is commutative when the girth of $M$ is at least 4 (the \textit{girth} of a matroid is the size of its smallest circuit). Using the disjoint automorphism criterion, we may find a matroid $M$ of a given girth such that the underlying $\ast$-algebra of $\QAutp M \pNB$ is noncommutative.  

\begin{example}
    Fix $r\geq 3$ and let $M$ be the rank $r$ matroid on the set $E = \{1,2,\ldots, r+2\}$ whose nonbases are
    \begin{equation*}
        \NB(M) = \left\{ E\setminus \{1,2\}, E\setminus \{3,4\}  \right\}.
    \end{equation*}
    By \Cref{prop:examples+noncommutative} we see that the underlying $\ast$-algebra of $\QAutp M \NB$ is noncommutative. In contrast, when $r\geq 4$, we have that $\QAut \pB M$ is commutative since the girth of $M$ is $r$.
    We also have that $\QAut \pB M$ is commutative when $r=3$, as we now demonstrate.  

    When $r=3$, the nonbases of $M$ are $\{1,2,5\}$ and $\{3,4,5\}$ (this is the cycle matroid of a graph obtained by removing a single edge from the complete graph on $4$ vertices). It suffices to show that
     $w_{ax}w_{by} = w_{by}w_{ax}$,
    where $a\neq b$ and $x\neq y$. 
    First, suppose that $a\in \{1,2\}$ and $b \in \{3,4\}$, or vice versa. To simplify the notation, consider the case where $a=1$ and $b=3$ as the other cases are similar. Then, in $\QAut{\NB}{M}$, we have
    \begin{align*}
        w_{1x}w_{3y}w_{1x} =
        w_{1x}w_{3y}(1 - w_{2x} - w_{3x} - w_{4x} - w_{5x}) =
        w_{1x}w_{3y}, 
    \end{align*}
where $w_{1x}w_{3y}w_{3x}=0$ since $x\neq y$ and $w_{1x}w_{3y}w_{jx}=0$ for $j\in \{2,4,5\}$ since $(1,3,j)$ is a basis tuple in those cases.
As $w_{1x}w_{3y}w_{1x}$ is self-adjoint, we have that $w_{1x}w_{3y} = w_{3y}w_{1x}$. 

Now we may suppose that $a,b \in \{1,2,5\}$ or $a,b\in \{3,4,5\}$. Without loss of generality, say $a=1$ and $b=2$. Then  
\begin{align*}
    w_{1x}w_{2y}w_{1x} =
    w_{1x}w_{2y}(1 - w_{2x} - w_{3x} - w_{4x} - w_{5x}) =
    w_{1x}w_{2y} - w_{1x}w_{2y}w_{5x}
\end{align*}
and
\begin{align*}
    w_{1x}w_{2y}w_{5x} = w_{1x}(1 - w_{1y} - w_{3y} - w_{4y} - w_{5y})w_{5x} = 0,
\end{align*}
since once again $w_{1x}w_{ky}w_{jx}=0$ either because $x\neq y$ or $(1,k,j)$ is a basis tuple for appropriate $k,j$. 
So $w_{1x}w_{2y} = w_{1x}w_{2y}w_{1x}$, and the latter is evidently self-adjoint. Therefore $w_{1x}w_{2y} = w_{2y}w_{1x}$. This completes the proof that $\QAut{\NB}{M}$ is commutative.   
\end{example}

\section*{Acknowledgements}
We thank Michael Joswig and the DM/G group at TU Berlin for hosting S. Schmidt and we thank the Department of Mathematical Sciences at the University of Nevada, Las Vegas for hosting M. Wack during various stages in this project. 
We thank Fabian Lenzen for helpful conversations. 
S. Schmidt acknowledges support by the Deutsche Forschungsgemeinschaft (DFG, German Research Foundation) under Germany’s Excellence Strategy - EXC 2092 CASA - 39078197.

\bibliographystyle{abbrv}
\bibliography{./src/bibliographie.bib}
\label{sec:biblio}

\end{document}